\documentclass[12pt]{amsart}
\usepackage{amsmath,amssymb,amsfonts,amsthm,graphicx}
\usepackage{pinlabel}
\usepackage[all]{xy}

\usepackage[usenames,dvipsnames]{color}

\usepackage[usenames,dvipsnames]{color}

\usepackage{hyperref}

\newtheorem{theorem}{Theorem}
\newtheorem{conjecture}[theorem]{Conjecture}

\theoremstyle{definition}
\newtheorem{definition}[theorem]{Definition}
\newtheorem{corollary}[theorem]{Corollary}

\newtheorem{example}[theorem]{Example}
\theoremstyle{remark}
\newtheorem{remark}{Remark}

\theoremstyle{plain} 
 
 \newtheorem*{mques}{Motivating Questions}
  \newtheorem{strat}{Strategy}

\newcommand{\Z}{\mathbb{Z}}
\newcommand{\R}{\mathbb{R}}

\newcommand{\F}{\mathbb{F}}

\newcommand{\alg}{\mathcal{A}}

\newcommand{\e}{\epsilon}
\newcommand{\leg}{\Lambda}

\begin{document}

\title{Constructions of Lagrangian Cobordisms}

\author[]{Sarah Blackwell} \address{ } \email{} 
\author[]{No\'emie Legout } \address{ } \email{} 
\author[]{Caitlin Leverson } \address{ } \email{} 
 \author[]{Ma\"ylis Limouzineau}
 \author[]{Ziva Myer}
 \author[]{Yu Pan}
 \author[]{Samantha Pezzimenti}
 \author[]{Lara Simone Su\'arez }
 \author[]{Lisa Traynor}


\begin{abstract} Lagrangian cobordisms between Legendrian knots arise in Symplectic Field Theory and impose an interesting and not well-understood
relation on Legendrian knots. 
There are some known ``elementary'' 
building blocks for Lagrangian cobordisms that  are smoothly the attachment of $0$- and $1$-handles.
An important question is whether every pair of non-empty Legendrians that are related by a connected Lagrangian
cobordism can be related by a ribbon Lagrangian cobordism, in particular one that is ``decomposable'' into a composition of  these elementary building blocks. We will describe 
these and other combinatorial building blocks as well as some geometric methods, involving the theory of satellites, to construct Lagrangian cobordisms. We will then survey some known
results, derived through Heegaard Floer Homology and contact surgery, that may provide a pathway to proving the existence of nondecomposable (nonribbon) Lagrangian cobordisms.
 \end{abstract}

\maketitle

\section{Introduction}

A contact manifold is an odd-dimensional manifold $Y^{2n+1}$ together with a 
maximally non-integrable hyperplane distribution $\xi$.  In a contact manifold, {\it Legendrian submanifolds} play a central role. 
These are the maximal integral submanifolds of $\xi$: $\leg^{n}$ such that  $T_{p}\leg \subset \xi$, for all $p  \in \Lambda$. 
In general, Legendrian submanifolds are plentiful and easy to construct.  
In this article we will restrict our attention to the contact manifold $\R^3$ with its standard contact structure 
$\xi = \ker \alpha$, where $\alpha = dz - ydx$.  In this setting, 
every smooth knot  or link has an infinite number of non-equivalent Legendrian 
representatives. More background on Legendrian knots is given in Section~\ref{sec:background}.

The even-dimensional siblings of contact manifolds  are symplectic manifolds.  These are even-dimensional manifolds $M^{2n}$
equipped with a closed, non-degenerate $2$-form $\omega$. In symplectic manifolds, {\it Lagrangian submanifolds} play a central role.  Lagrangian submanifolds are the maximal dimensional submanifolds where $\omega$ vanishes on the tangent spaces: $L^{n}$ such that $\omega|_{L} = 0$.
  When the symplectic manifold is exact, $\omega = d \lambda$, it is important
to understand the more restrictive subset of {\it exact} Lagrangians: these are submanifolds where $\lambda|_{L}$ 
is an exact $1$-form.  Geometrically, $L$ exact means that for any closed curve $\gamma \subset L$, $\int_\gamma \lambda = 0$.
 In this article, we will restrict
our attention to a symplectic manifold that is symplectomorphic to $\R^{4}$ with its standard symplectic
structure $\omega_{0} = \sum dx_{i} \wedge dy_{i}$.  In contrast to
Legendrians, Lagrangians are scarce.  For example, in $\R^4$ with its standard symplectic
structure, the torus is the only closed surface that will admit a Lagrangian embedding
into $\R^4$.    A famous theorem of Gromov  \cite{Gro} states that there are {\it no} closed, exact Lagrangian submanifolds of $\R^4$.

 There has been a great deal of
recent interest in a certain class of non-closed, exact Lagrangian submanifolds, known as
{\it Lagrangian cobordisms}.  These Lagrangian submanifolds live in the symplectization
of a contact manifold and have cylindrical ends over Legendrians.  In this article,
we will focus on exact, orientable Lagrangian cobordisms from the Legendrian $\Lambda_-$ to the Legendrian $\Lambda_+$ that live in the symplectization of $\R^{3}$; this
symplectization is $\R \times \R^{3}$ equipped with the exact symplectic form $\omega = d(e^{t} \alpha)$, where $t$ is the coordinate on $\R$ and  $\alpha = dz - ydx$ is the standard contact form on $\R^{3}$.
See Figure~\ref{fig:cob} for a schematic
picture of a Lagrangian cobordism  and Definition~\ref{defn:cobordism} for a formal definition.  Such Lagrangian cobordisms were first introduced 
in Symplectic Field Theory (SFT) \cite{EGH}: in relative SFT,  we get a category whose objects are Legendrians and
whose morphisms are Lagrangian cobordisms. {\it Lagrangian fillings} occur when $\leg_{-} = \emptyset$ and are key objects in the Fukaya category, which is an important invariant of symplectic four-manifolds. 
A {\it Lagrangian cap} occurs when $\leg_{+} = \emptyset$.  

A   basic question  tied to understanding the general existence and behavior of Lagrangian submanifolds
   is to understand the existence of Lagrangian cobordisms: {\it Given two Legendrians
 $\leg_\pm$, when does there exist a Lagrangian cobordism from $\leg_-$ to $\leg_+$?}
 There are known to be a number of obstructions to this relation on Legendrian submanifolds coming from both  classical and 
 non-classical invariants of the Legendrians $\leg_\pm$. Some of these obstructions are described in Section~\ref{ssec:obstructions}.
 To complement the obstructions,  there are some known constructions.  For example,  it is well known \cite{EliGro, Cha, EkHoKa} that there exists a Lagrangian
 cobordism between Legendrians $\leg_\pm$ that differ by Legendrian isotopy.
 In addition, by  
 \cite{EkHoKa, C:slice},
 it is known that there exists a Lagrangian cobordism from $\leg_{-}$ to $\leg_{+}$  if
 $\leg_-$ can be obtained from $\leg_{+}$ by  a ``pinch'' move or 
 if $\leg_{+} = \leg_{-} \cup U$, where $U$ denotes a Legendrian unknot with maximal Thurston-Bennequin number of $-1$
  that is contained in the complement of a ball containing $\leg_{-}$.
 Topologically, between these slices, the cobordism changes by a saddle move (1-handle) and the addition of a  local minimum (0-handle); see Figure~\ref{fig:pinch-min}.
 It is important to notice that there is {\it not} an elementary move corresponding to a local maximum (2-handle) move.
 By stacking these individual cobordisms obtained from isotopy, saddles, and minimums, one obtains
 what is commonly referred to as a  {\it decomposable} Lagrangian cobordism.   Through these moves, it is easy to construct Lagrangian cobordisms
 and fillings; see an example in  Figure~\ref{fig:tref}.

 \begin{figure}[!ht]
 \labellist
 \small
 \pinlabel $(A)$ at 45 -10
 \pinlabel $(B)$ at 235 -10
 \pinlabel $\Lambda_+$ at -10 100
 \pinlabel $\Lambda_-$ at -10 20
 \pinlabel $\Lambda_+$ at 170 110
 \endlabellist
\includegraphics[width=1in]{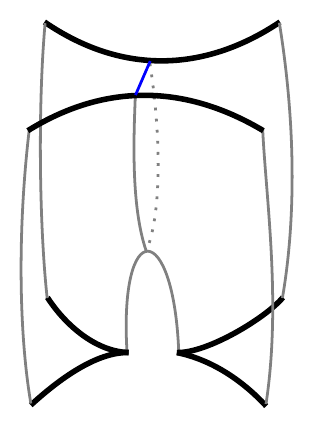}\hspace{1in}
\includegraphics[width=1in]{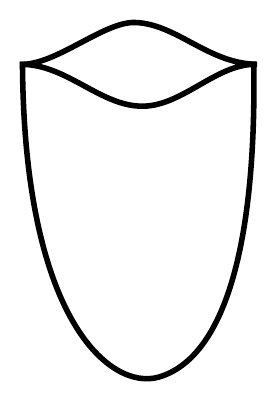}
\vspace{0.1in}

\caption{(A) The pinch move on $\leg_{+}$  produces a Lagrangian saddle.  (B) $\leg_{+}$ obtained by introducing an unknotted component to $\leg_{-}$
  corresponds to the Lagrangian cobordism having a local min.}
\label{fig:pinch-min}
\end{figure}
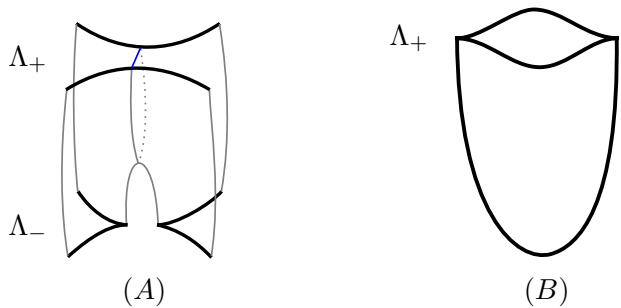

 Towards understanding the existence of Lagrangians, it is natural to ask:
{\it Does there exist a Lagrangian cobordism from $\leg_{-}$ to $\leg_{+}$ if and only if there exists a decomposable Lagrangian cobordism from $\leg_{-}$ to $\leg_{+}$?}
We know the answer to this question is ``No'':  by studying the ``movies'' of the not necessarily Legendrian slices of a Lagrangian.
Sauvaget, Murphy, and Lin \cite{Sau, Lin} have shown that there exists a genus two Lagrangian cap of %$U_{-3}$, which
 the Legendrian unknot with Thurston-Bennequin number equal to $-3$ and rotation number $0$.  The {\it Lagrangian diagram moves} used by \cite{Lin} to
 construct a  Lagrangian cap are described in Section~\ref{ssec:lin}.
% It is not currently understood the necessity of a local maximum when $\leg_{+} \neq \emptyset$.  
The necessity of a local maximum when $\leg_{+} \neq \emptyset$ is not currently understood.

To formulate some precise motivating questions, we will use 
{\it ribbon cobordism} to denote a $2n$-dimensional manifold that can be built from $k$-handles with $k \leq n$.  This idea of restricting the handle index is well known in symplectic topology:
Eliashberg \cite{CE, Oan} has shown that any $2n$-dimensional Stein manifold admits a handle decomposition with handles of dimension at most $n$, and thus any $2n$-dimensional  Stein cobordism between closed, $(2n-1)$-dimensional contact manifolds must be ribbon.  Working in the relative setting  with submanifolds and using the handle decomposition from 
%viewing the index of a handle with respect to 
the ``height'' function given
by the $\mathbb R$ coordinate on $\mathbb R \times \mathbb R^{3}$,
we see that all decomposable $2$-dimensional Lagrangian cobordisms between $1$-dimensional Legendrian submanifolds are ribbon cobordisms. 
We are led to the following natural questions.

\begin{mques}  Suppose $\leg_{+} \neq \emptyset$ and there exists a connected Lagrangian cobordism $L$ from $\leg_{-}$ to $\leg_{+}$.  Then:
 \begin{enumerate}
\item   \label{ques:decomp} Does there exist a decomposable Lagrangian cobordism from $\leg_{-}$ to $\leg_{+}$? 
\item   Does there exist a ribbon Lagrangian cobordism from $\leg_{-}$ to $\leg_{+}$?
\item Is $L$ Lagrangian isotopic to a ribbon and/or decomposable Lagrangian cobordism?
\end{enumerate}
\end{mques}

There are some results known about Motivating Question (3) for the special case of the simplest Legendrian unknot.
If $U$ denotes the Legendrian unknot with Thurston-Bennequin number $-1$,  it is known
that every (exact) Lagrangian filling is orientable \cite{Ritter},  and there is a unique (exact, orientable) Lagrangian filling of $U$ up to compactly supported Hamiltonian isotopy \cite{EP}.  
Moreover,  any Lagrangian cobordism from $U$  to $U$ is 
Lagrangian isotopic, via a compactly supported Hamiltonian isotopy, to one in a countable collection given by the trace of a Legendrian isotopy induced by a rotation \cite{CDRGG-concord}.

Motivating Questions (1) and (2) are closely related and have deep ties to important questions in topology.
Observe that a ``yes'' answer to (1) implies a ``yes'' to (2):
 if the existence of a Lagrangian cobordism
implies the existence of a decomposable Lagrangian cobordism, then we also know the existence of a ribbon cobordism.
  Also note that when $\leg_{+}$ is topologically a slice knot and $\leg_{-} = \emptyset$, (2) is a symplectic version of the topological Slice-Ribbon conjecture:
 {is every Lagrangian slice disk a ribbon disk?} 
 Cornwell, Ng, and Sivek conjecture that the answer to Motivating Question (1) and (3) is ``No'': 
using the theory of satellites, we know that there is a Lagrangian concordance between $\leg_{\pm}$ shown in Figure~\ref{fig:satellite}, and in \cite[Conjecture 3.3]{CNS}
it is 
conjectured that the concordance between the pair is not decomposable.

\begin{figure}[!ht]
\labellist
\pinlabel $\Lambda_+$ at -20 500
\pinlabel $\Lambda_-$ at -20 100
\endlabellist

\includegraphics[width=2in]{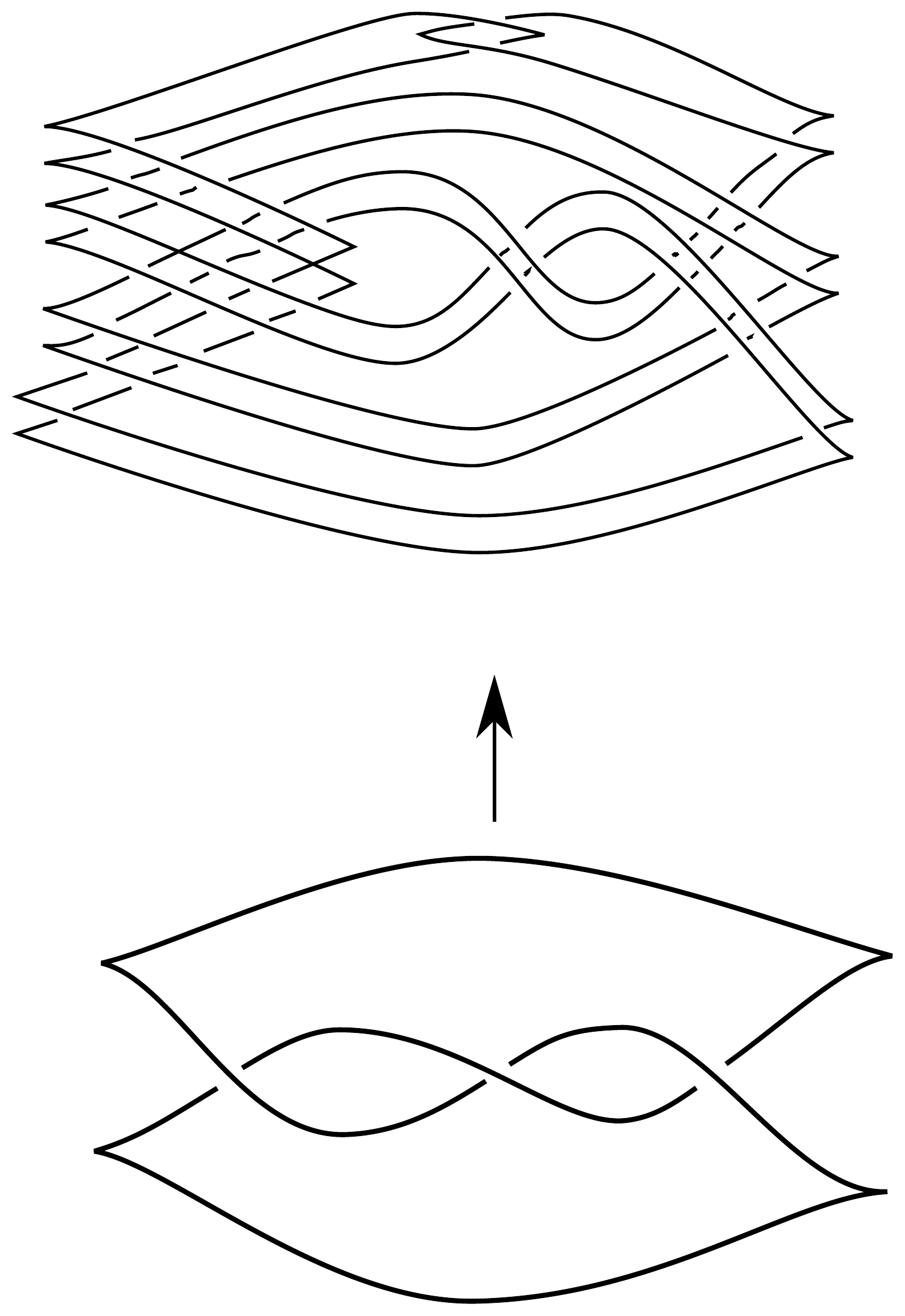}

\caption{There is a Lagrangian concordance between these Legendrian knots that is conjectured to be non-decomposable. Here
$\leg_{-}$ is a Legendrian trefoil  and $\leg_{+}$ is a Legendrian  Whitehead double of  $m(9_{46})$.
}
\label{fig:satellite}
\end{figure}

 Very recently, Roberta Guadagni has discovered additional combinatorial moves that can be used to construct a ``movie,'' meaning a sequence of slice pictures, of a Lagrangian cobordism; Figure~\ref{fig:R-move} illustrates one of these tangle moves. 
With one of Guadagni's moves, it is possible to construct a movie of a Lagrangian cobordism between the Legendrians pictured in Figure~\ref{fig:satellite}; see Figure~\ref{fig:guad-movie}.
Guadagni's moves are ``geometric'': they are developed through proofs similar to those used in 
the satellite procedure, and thus the handle attachments involved in the cobordism are not obvious.
In particular, at this point it is not known if Guadagni's tangle moves are independent from the decomposable moves.

This survey article is organized as follows.
In Section ~\ref{sec:background}, we  provide some background on Legendrians and Lagrangians, formally define Lagrangian cobordisms,
 and summarize known obstructions to the existence of Lagrangian cobordisms.  In Section~\ref{sec:construct}, we describe three ``combinatorial''  ways
 to construct Lagrangian cobordisms, and in Section~\ref{sec:satellites}, we describe  more abstract ``geometric'' ways to construct Lagrangian concordances
 and cobordisms through satellites. 
  Then in Section~\ref{sec:non-decomp-candidates}, we describe some potential pathways -- through the theory of rulings, Heegaard-Floer homology, and contact surgery --
 to potentially show the existence of Legendrians that are Lagrangian cobordant but are not related by a  decomposable Lagrangian cobordism.
 
 \smallskip
 
\noindent
{\bf Acknowledgements:} This project was initiated at the workshop Women in Symplectic and Contact Geometry and Topology (WiSCoN) that took place at ICERM in July 2019.  The authors
thank the NSF-HRD 1500481 - AWM ADVANCE grant for funding this workshop.
Leverson was supported by NSF postdoctoral fellowship DMS-1703356. We thank Emmy Murphy for suggesting and encouraging us to work on this project.  
 In addition, we thank John Etnyre, Roberta Guadagni,  Tye Lidman, Lenny Ng, Josh Sabloff, and B\"ulent Tosun for useful conversations related to this project.

\section{Background} \label{sec:background}

\subsection{Legendrian Knots and Links} \label{ssec:legendrians}
In this section, we give a very brief introduction to Legendrian submanifolds in $\R^3$ and their invariants. 
More details can be found, for example, in the survey paper \cite{Etnyre}.

In $\R^3$, the {\bf standard contact structure} $\xi$ is a $2$-dimensional plane field given by the kernel of the $1$-form $\alpha=dz-ydx$.
In $(\R^3, \xi=\ker \alpha)$, a {\bf Legendrian knot} is a knot in $\R^3$ that is tangent to $\xi$ everywhere.
A useful way to visualize a Legendrian knot is to project it  from $\R^3$ to $\R^2$. 
There are two useful projections: the {\bf  Lagrangian projection}
$$
\begin{array}{rccc}
\pi_L: & \mathbb{R}^3 & \rightarrow & \mathbb{R}^2 \\
& (x,y,z) & \mapsto &(x,y),
\end{array}
$$
as well as the {\bf front projection}
$$
\begin{array}{rccc}
\pi_F: & \mathbb{R}^3 & \rightarrow & \mathbb{R}^2 \\
& (x,y,z) & \mapsto &(x,z).
\end{array}
$$
An example of  a Legendrian trefoil is shown in Figure~\ref{fig:trefoil}.

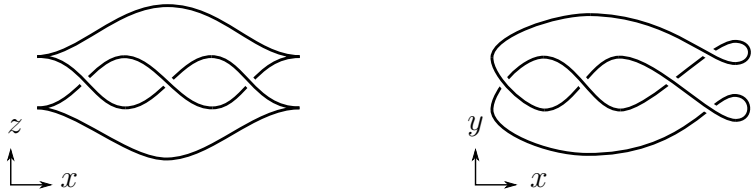
\begin{figure}[!ht]
\labellist
\pinlabel $x$ at 40 10
\pinlabel $z$ at 10 40
\pinlabel $x$ at 300 10 
\pinlabel $y$ at 265 40
\endlabellist
\includegraphics[width=4in]{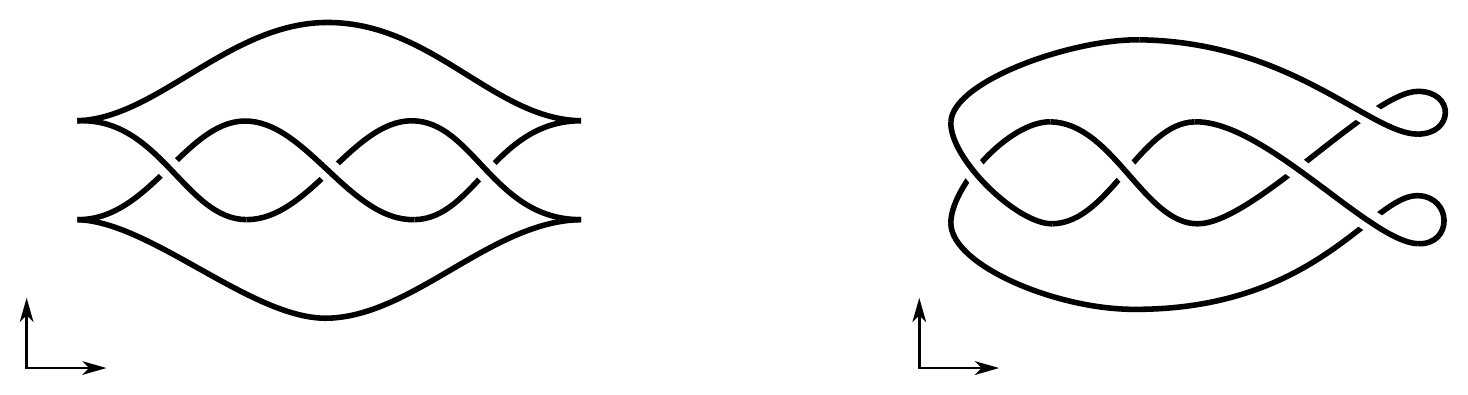}
\caption{The front projection (left) and the Lagrangian projection (right) of a Legendrian trefoil.}
\label{fig:trefoil}
\end{figure}

Legendrian submanifolds are equivalent if they can be connected by a $1$-parameter family
of Legendrian submanifolds.  In fact, for each topological knot type there are infinitely many different 
 Legendrian knots. 
Indeed, we can stabilize a Legendrian knot (as shown in Figure~\ref{fig:stab}) to  get another Legendrian knot of  the same topological knot type. 
We can see that these are not Legendrian equivalent using Legendrian invariants.

\begin{figure}[!ht]
\includegraphics[width=3in]{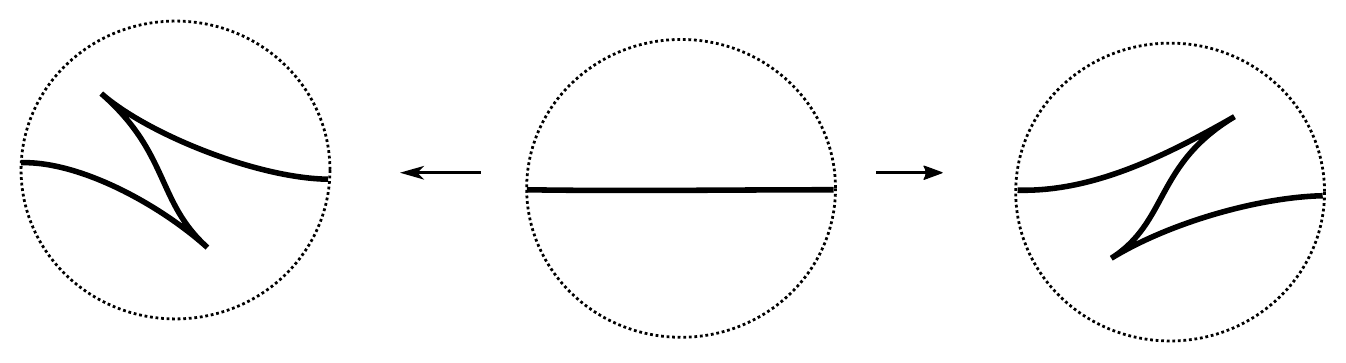}
\caption{Two ways to stabilize a Legendrian knot in front projection.}
\label{fig:stab}
\end{figure}

Two useful classical invariants of Legendrian knots $\Lambda$ are the {Thurston-Bennequin number} $tb(\Lambda)$ and the {rotation number} $r(\Lambda)$.
They can be computed easily from front projections.
Given the front projection of a Legendrian knot or link $\Lambda$,
the {\bf Thurston-Bennequin number} is $$tb(\Lambda)= \mbox{writhe}(\pi_F(\Lambda))-\#(\mbox{right cusps}),$$ where the writhe is the number of crossings counted with sign.
Once the Legendrian knot is equipped with an orientation, the {\bf rotation number} is 
$$r(\Lambda)=\frac{1}{2}\Big(\#(\mbox{down cusps})-\#(\mbox{up cusps})\Big).$$
One can use these two invariants to see that  stabilizations change the Legendrian knot type.

In future sections, we will not assume that our Legendrians $\leg_{\pm}$ come equipped with an orientation. In our Motivating Questions described in Section 1, our Lagrangian cobordisms are always orientable, so the existence of a Lagrangian cobordism from $\leg_{-}$ to $\leg_{+}$ will induce orientations on $\leg_{\pm}$.

 There are many powerful non-classical invariants that can be assigned to a Legendrian knot.  Although this will not be a focus of this paper, we will give a brief
 description of some of these invariants.  One important invariant stems from {\it normal  rulings}, defined independently by Chekanov and Pushkar \cite{CP} and Fuchs \cite{Fuchs}.  
 A count of normal  rulings leads to {\it ruling polynomials} \cite{CP}; more details will be discussed in Section~\ref{sec:rul}.
 Through the closely related theory of { generating families}, one can also associate invariant polynomials that record the dimensions of {\it generating family homology groups}
 \cite{lisa:links, lisa-jill, f-r, ST:obstruct}. In addition, through the theory of pseudo-holomoprhic curves, one can associate to a Legendrian $\Lambda$ a {\it differential graded algebra (DGA)}, $\mathcal A(\Lambda)$ \cite{Che, Eli}.
 An {\it augmentation} is a DGA map from $\mathcal A(\Lambda)$ to a field. 
The count of augmentations is closely related to the count of ruling polynomials \cite{Fuchs, NR, NS}.
 %REF, REF 
 Augmentations can be used to construct finite-dimensional {\it linearized contact homology groups} \cite{Che}, which are
often known to be isomorphic to the generating family homology groups \cite{f-r}.  In addition, there are invariants for Legendrian knots coming from {\it Heegaard Floer Homology}
\cite{LOSS} \cite{OST}.

\subsection{Lagrangian Cobordisms}
Lagrangian cobordisms between Legendrian submanifolds always have ``cylindrical ends'' over the Legendrians, but other conditions vary: sometimes it is specified that the Lagrangian
is exact,  is embedded (or immersed), is orientable, or has a fixed Maslov class.    In this paper, a Lagrangian cobordism is always exact, embedded, and orientable. 
 
\begin{definition} \label{defn:cobordism} Let $\Lambda_{\pm}$ be two Legendrian knots or links in $(\R^3, \xi=\ker \alpha)$. A {\bf Lagrangian cobordism $L$ from  $\Lambda_-$ to $\Lambda_+$}
is an embedded, orientable Lagrangian surface in the symplectization $(\R\times\R^3, d(e^t\alpha))$ 
	   such that for some $N > 0$,
\begin{enumerate}
		\item  $L \cap ([-N,N]\times\R^3)$ is compact,
		\item  $L \cap ((N,\infty)\times\R^3)=(N,\infty)\times \Lambda_{+}$, 
		\item   $L \cap ((-\infty,-N)\times\R^3)=(-\infty,-N)\times \Lambda_{-}$, and 
		\item there exists  a function $f: L \to \R$  and constant numbers $\frak c_\pm$ such that 
		$e^t\alpha|_{TL} = df$,
		 where $f|_{(-\infty, -N) \times \Lambda_{-}} = \frak c_{-}$, and $f|_{(N, \infty) \times \Lambda_{+}} = \frak c_{+}$.  
\end{enumerate}
A {\bf Lagrangian filling of $\Lambda_+$} is a Lagrangian cobordism with $\leg_- = \emptyset$; a  {\bf Lagrangian cap of $\Lambda_-$} is a Lagrangian cobordism with $\leg_+ = \emptyset$.
 A {\bf Lagrangian concordance} occurs when $\Lambda_\pm$ are knots and $L$ has genus $0$.
 \end{definition}
 
 Figure~\ref{fig:cob} is a schematic representation of a Lagrangian cobordism.
\begin{figure}[!ht]
\labellist
\pinlabel $t$ at -10 370
\pinlabel $N$ at -10 260 
\pinlabel $\Lambda_+$ at 350 300 
\pinlabel $\Lambda_-$ at 350 110
\pinlabel $L$ at 350 200
\pinlabel $-N$ at -20 70
\endlabellist
\includegraphics[width=2in]{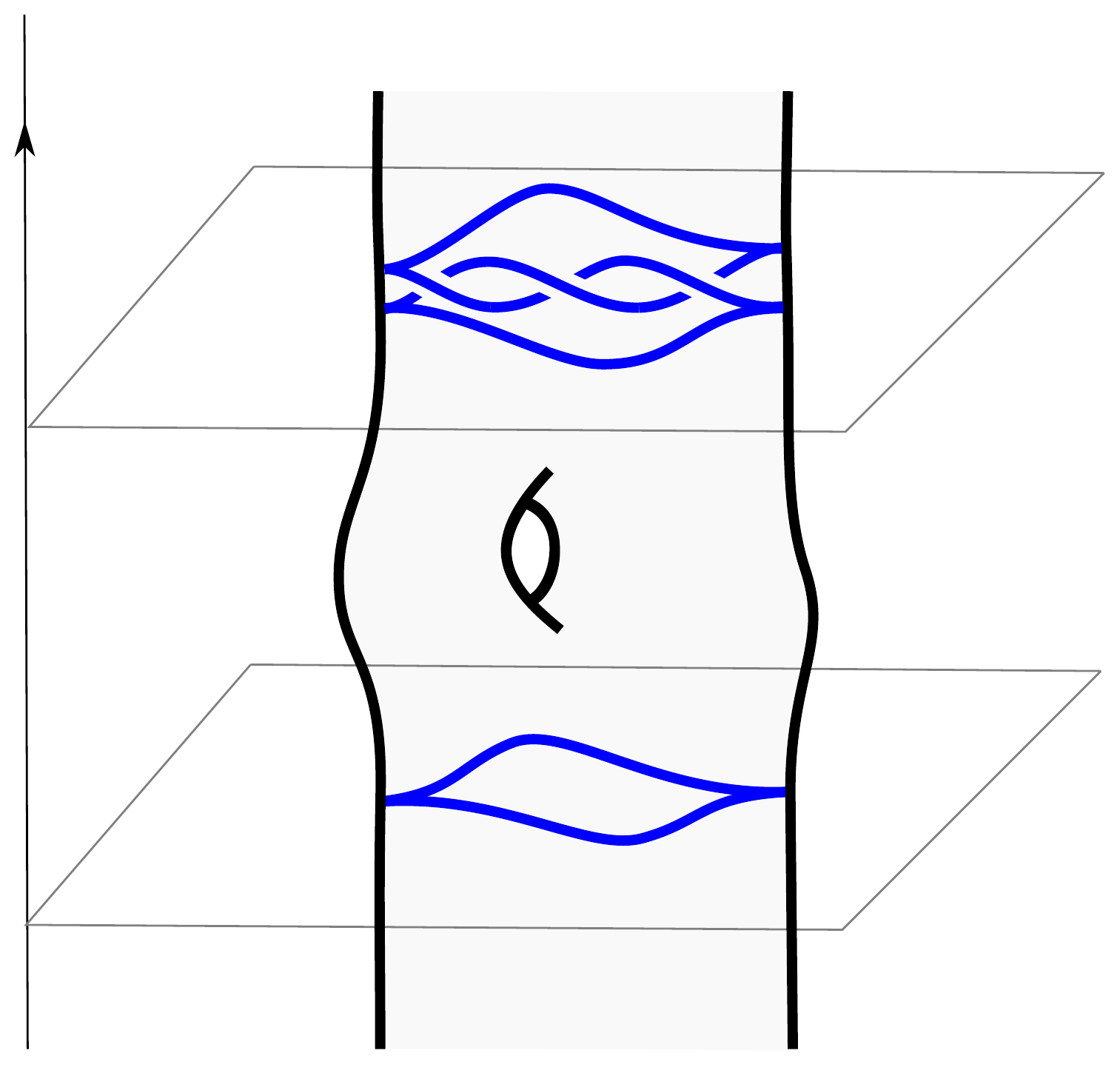}
\caption{A Lagrangian cobordism from $\Lambda_-$ to $\Lambda_+$.  
}
\label{fig:cob}
\end{figure}
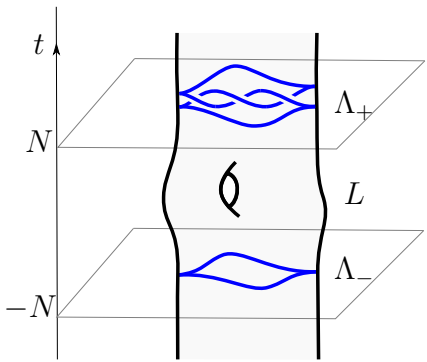

 \begin{remark} 
 In condition (4) of Definition~\ref{defn:cobordism}, the fact that $\leg_\pm$ are Legendrian will guarantee that $f_{\pm}$ will be locally constant.  Using this, it follows that
  any genus zero Lagrangian surface that is cylindrical over 
 Legendrian knots will be exact.  When $\leg_\pm$ have multiple components, one needs to 
 check that the constant does not vary: this condition guarantees the exactness of the Lagrangian cobordism obtained by  ``gluing'' together Lagrangian cobordisms \cite{chantraine:disconnected-ends}.
 \end{remark}
 
 \begin{remark}
 In contrast to topological cobordisms,  Lagrangian cobordisms form a non-symmetric relationship on Legendrian knots \cite{Cha2}.
In this article we will always denote   the direction of increasing
 $\mathbb R_t$ coordinate  by an arrow.   
 \end{remark}

\subsection{Obstructions to Lagrangian Cobordisms} \label{ssec:obstructions}
The focus of this paper is on constructing Lagrangian cobordisms between two given Legendrians $\leg_{\pm}$. In the smooth world, any two knots are related by a smooth cobordism,
 but in this more restrictive Lagrangian world, 
there are a number of obstructions that are important to keep in mind when trying to explicitly construct Lagrangian cobordisms.   
  Here we mention a few that come from
classical and non-classical invariants of the Legendrians $\leg_{\pm}$.
  
{\bf Obstructions:}  
\begin{enumerate}
\item If there exists a Lagrangian cobordism  of genus $g$ between $\leg_{-}$ and $\leg_{+}$, then there must exist a smooth cobordism of genus $g$ between the smooth knot types of $\leg_-$ and $\leg_+$.  Thus any obstruction of a smooth genus $g$ cobordism between $\leg_{-}$ and $\leg_{+}$ would obstruct a Lagrangian genus $g$ cobordism.
\item Since  there are no closed, exact Lagrangian surfaces \cite{Gro}, if there exists a Lagrangian cap (respectively, filling) for $\leg$, then there cannot exist a Lagrangian filling (respectively, cap) of $\leg$.
\item  As shown in  \cite{Cha}, if %$\leg_\pm$ are Legendrian knots and
there exists a Lagrangian cobordism $L$ from $\leg_{-}$ to $\leg_{+}$, then
$$
r(\leg_{-}) = r(\leg_{+}) \quad \text{ and } \quad
tb(\leg_{+})-  tb(\leg_{-})= -\chi(L).
$$
In particular, if a Legendrian knot $\leg$ admits a Lagrangian filling or cap, then $r(\leg) = 0$.
 Also, combining this equality on $tb$  and the slice-Bennequin inequality \cite{rudolph:slice-tb},  we see that,
  when $\Lambda$ is a single component knot, 
 if there exists a Lagrangian cap $L$ of  $\leg$, then $tb(\leg) \leq -1$ and $g(L) \geq 1$.  
 \item If there exists a Maslov $0$ (\cite{EESu}) Lagrangian cobordism $\Sigma$ from $\leg_{-}$ to $\leg_{+}$, and $\leg_{-}$ has an augmentation, then
\begin{enumerate}
\item $\#Aug(\Lambda_+;\F_{2})\ge\#Aug(\Lambda_-;\F_{2})$, where  $\F_{2}$ is the finite field of two elements, and $\#Aug(\Lambda;\F_{2})$ denotes the number of augmentations of $\Lambda$ to $\F_{2}$ up to DGA homotopy \cite{Pan1, WIG}, and
\item the ruling polynomials $R_{\Lambda_{\pm}}(z)$ (see Section~\ref{sec:rul} for definitions)
satisfy 
$$R_{\Lambda_-}(q^{1/2}-q^{-1/2})\leq q^{-\chi(\Sigma)/2}R_{\Lambda_+}(q^{1/2}-q^{-1/2}),$$
 for any $q$ that is a power of a prime number \cite{Pan1}.
\end{enumerate}
\item If $\Lambda$ admits a Maslov $0$ Lagrangian filling $L$, and if $\e_L$ denotes the augmentation of $\Lambda$ induced by $L$, then $LCH^{k}_{\e_L}(\Lambda)\cong H_{n-k}(L)$, which is known as the Ekholm-Seidel isomorphism \cite{Ekholm}, and whose proof was completed by Dimitroglou Rizell in \cite{DR}. More generally, if there is a cobordism from $\Lambda_-$ to $\Lambda_+$, and if $\Lambda_-$ admits an augmentation, then \cite{CDRGG-cobord} provides several long exact sequences relating the homology of the cobordism and the Legendrian contact (co)homologies of its Legendrian ends. A version of this isomorphism and these long exact sequences using generating families are given in \cite{ST:obstruct}.
\item  
If $\Lambda$ admits an augmentation, $\Lambda$ does not admit a Lagrangian cap, as the augmentation implies the non-acyclicity of the DGA $\mathcal A(\Lambda)$ \cite[Theorem 5.5]{EES}, and from \cite[Corollary 1.9]{DR1} if a Legendrian admits a Lagrangian cap then its DGA $\mathcal A(\Lambda)$ (with $\Z_2$ coefficients) is acyclic.

\end{enumerate}
 There are additional obstructions, obtained through Heegaard Floer Theory, that can be used to obstruct Lagrangian concordances and cobordisms \cite{BS, GJ, BLW}.
Some of these will be discussed more in Section~\ref{ssec:Floer}.
 
\begin{remark} \label{rem:no-stab-obstruct}
Observe that the obstructions in (4) and (6) assume that  the bottom $\Lambda_-$ has an augmentation, and stabilized knots
will never have an augmentation.  
 It would be nice to have more obstructions when
$\leg_{-}$ is a stabilized knot. This might be possible using the theory of ``satellites'' described in Section~\ref{sec:sat}:  it is possible for the satellite of a stabilized Legendrian to admit an augmentation.  See Section~\ref{ssec:sat-obstruct} for more discussions in this direction.
 \end{remark}
 
 \section{Combinatorial Constructions of  Lagrangian Cobordisms} \label{sec:construct}

A convenient way of visualizing topological cobordisms is through ``movies'': a sequence of pictures  that represent %3-dimensional 
slices of the Lagrangian.  
In this section, we describe three known combinatorial ways to construct Lagrangian cobordisms through such an approach.  

\subsection{Decomposable Moves} \label{ssec:decomp-moves}

It is well known that if $\leg_-$ and $\leg_+$ are Legendrian isotopic, then there exists a Lagrangian cobordism from $\leg_-$ to $\leg_+$; see, for example, \cite{EliGro, Cha, EkHoKa}.
Isotopy, together with two types of handle moves, form the basis for decomposable Lagrangian cobordisms.

\begin{theorem}[\cite{EkHoKa, bst:construct}] If the front diagrams of two Legendrian links $\leg_-$ and $\leg_+$ are related by any of the following moves, there is a Lagrangian cobordism $L$ from $\leg_-$ to $\leg_+$.
\begin{description}
\item[Isotopy] There is a Legendrian isotopy between  $\leg_-$ and $\leg_+$; see Figures~\ref{fig:decomp}(a)-\ref{fig:decomp}(c) for Reidemeister Move I-III.
\item[$1$-handle] The front diagram of $\leg_{-}$ can be obtained from the front diagram of $\leg_{+}$ by ``pinching'' two oppositely-oriented strands; see Figure~\ref{fig:decomp}(d). We will also refer to this move as a ``Pinch Move.''
\item[$0$-handle] The front  diagram of $\leg_{-}$ can be obtained from the front diagram of $\leg_{+}$ by deleting a component of $\leg_{+}$ that is the front diagram of a standard Legendrian unknot $U$ with maximal Thurston-Bennequin number of $-1$ as long as there exist disjoint disks $D_{U}, D_{U^{c}} \subset \mathbb R_{xz}^{2}$ containing the $xz$-projection of $U$ and the other components of $\Lambda_{+}$, respectively.
Such an ``unknot filling'' can be seen in Figure~\ref{fig:decomp}(e).

% $\leg_+$ consists of the \ch{disjoint }union of 
%the front diagram of $\leg_- $ and a Legendrian unknot $U$ with maximal Thurston-Bennequin number; see Figure~\ref{fig:decomp}(e).
\end{description}
\end{theorem}

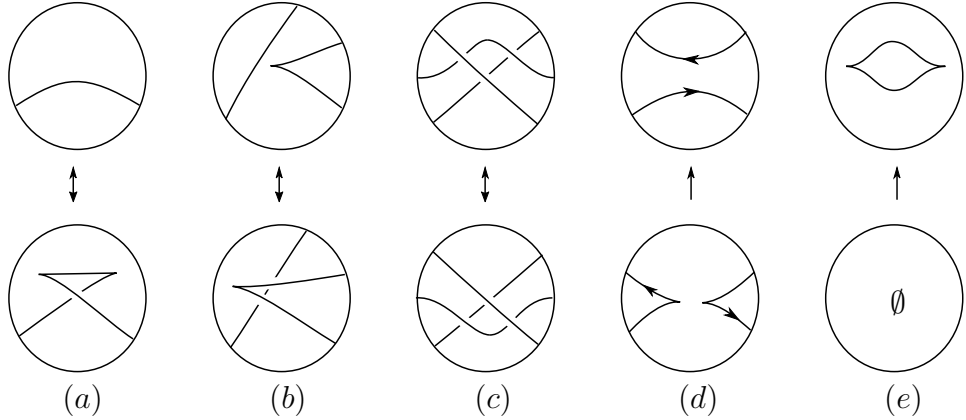
\begin{figure}[!ht]
\labellist
\pinlabel $(a)$ at 40 -15
\pinlabel $(b)$ at 150 -15
\pinlabel $(c)$ at 260 -15
\pinlabel $(d)$ at 370 -15
\pinlabel $(e)$ at 480 -15
\pinlabel $\emptyset$ at 477 38
\endlabellist
\includegraphics[width=5in]{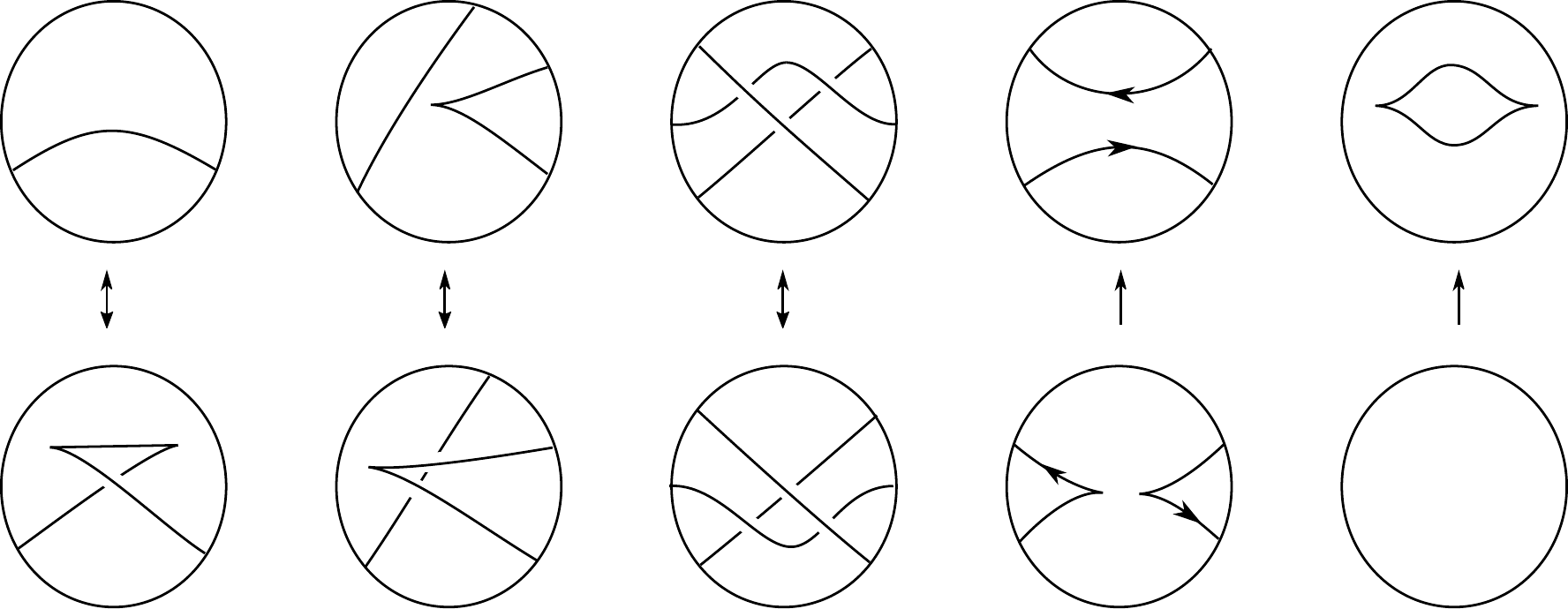}
\vspace{0.2in}

\caption{Decomposable moves in terms of front projections. 
Arrows indicate the direction of increasing $\mathbb R_t$ coordinate in the symplectization.  
The move in $(b)$ only shows the Reidemeister II
 move in the left cusp case, but there is an analogous move for the right cusp.}
\label{fig:decomp}
\end{figure}

\begin{definition}
A Lagrangian cobordism $L$ from $\leg_-$ to $\leg_+$ is called {\bf elementary} if it arises from isotopy, a single $0$-handle, or a single $1$-handle.  
A Lagrangian cobordism $L$ from $\leg_{-}$ to $\leg_{+}$ is {\bf decomposable} if it is obtained by stacking elementary Lagrangian cobordisms.
\end{definition} 

Observe that there is {\it not} an elementary
move corresponding to a $2$-handle (maximum).  Also note that the elementary $1$-handle (saddle) move can be used to connect two components or to split one component into two.

Decomposable cobordisms are particularly convenient as they are easy to describe in a combinatorial fashion, through a list of embedded Legendrian curves,
$$\Lambda_-=\Lambda_0\rightarrow \Lambda_1\rightarrow \cdots \rightarrow \Lambda_n=\Lambda_+,$$
where the front projection of the Legendrian $\Lambda_{i+1}$ is related to that of  $\Lambda_i$ by  isotopy or one of  the $0$-handle or $1$-handle moves.

\begin{example}\label{ex:tref} One can construct a Lagrangian filling of a positive Legendrian trefoil  with maximal Thurston-Bennequin number using the series of moves shown in  Figure~\ref{fig:tref}: a $0$-handle, followed by three Reidemeister I moves, followed by two $1$-handles (or pinch moves).
This gives a genus $1$ (orientable, exact) Lagrangian filling of this Legendrian trefoil.  Since we are assuming that Lagrangian fillings and caps are always exact, this implies that this trefoil cannot admit a Lagrangian cap; see Section~\ref{ssec:obstructions} Obstructions (2).
\begin{figure}[!ht]
\includegraphics[width=4in]{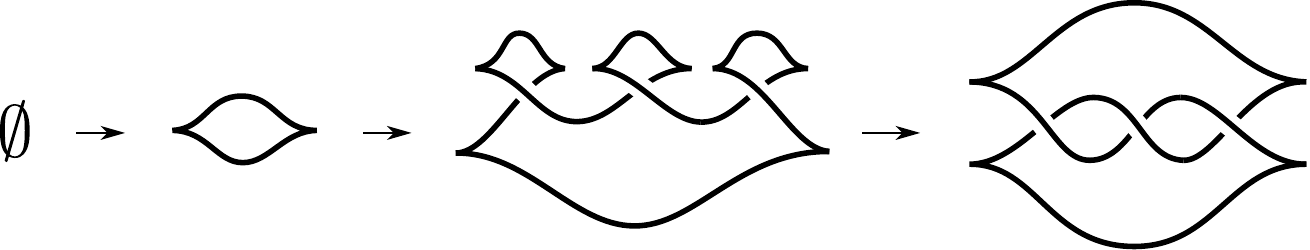}
\caption{A decomposable Lagrangian filling of a Legendrian trefoil. 
}
\label{fig:tref}
\end{figure}

\end{example}

\begin{example}  Using elementary moves, one can also construct a Lagrangian concordance from the unknot with $tb=-1$ to a Legendrian representative of the knot $m(9_{46})$, as shown on Figure~\ref{fig:m946}.  \begin{figure}[!ht]
\includegraphics[width=5in]{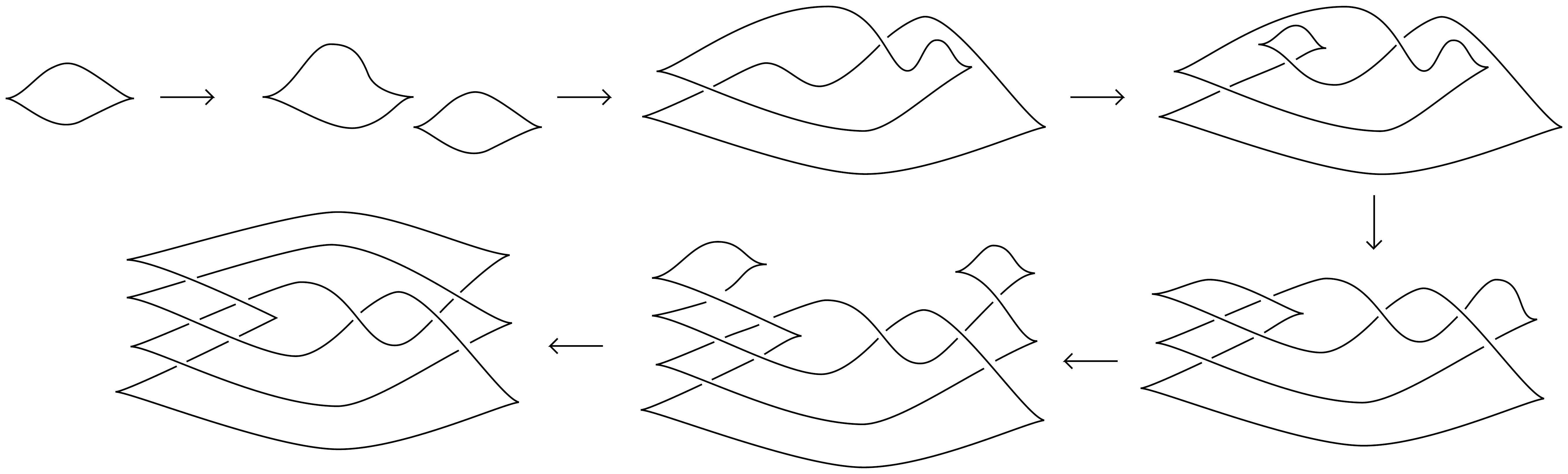}
\caption{A decomposable Lagrangian cobordism from a Legendrian unknot to a Legendrian $m(9_{46})$.  }
\label{fig:m946}
\end{figure}
\end{example}

\subsection{Guadagni Moves}\label{sec:Gmoves}

Very recently, Roberta Guadagni has discovered a new ``tangle'' move; see Figure~\ref{fig:R-move}. This is not a local move: there are some global requirements.  In particular, this move cannot
be applied if  all components of the tangle are contained in the same component of $\leg_{-}$: the component of
$\leg_{-}$  containing the blue strand must be different than the components containing the other strands of the tangle.

\begin{figure}[!ht]
\centerline{\includegraphics[height=1.2in]{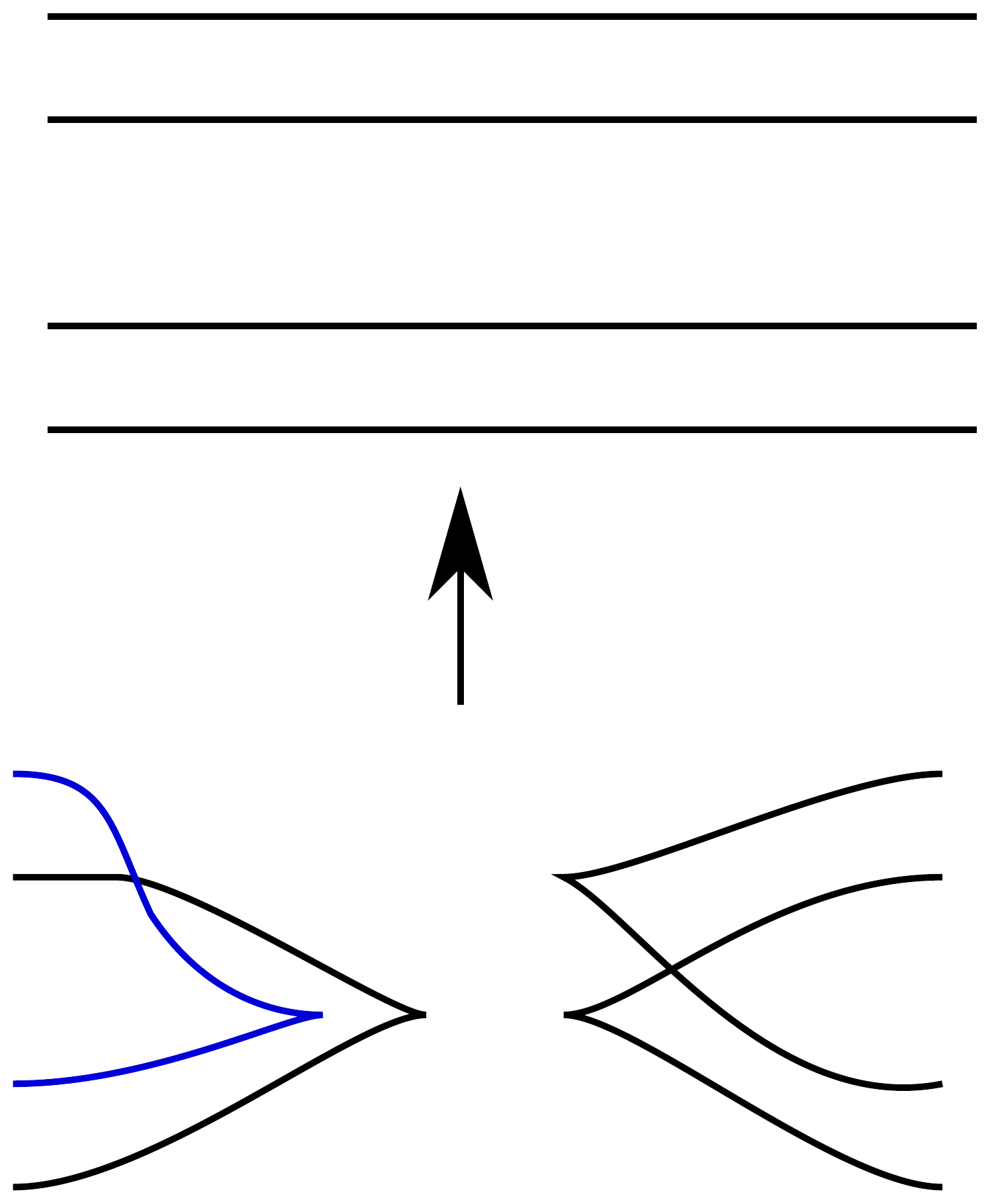}}
\caption{Under some global conditions, there exists a Lagrangian cobordism between these tangles.}
\label{fig:R-move}
\end{figure}

\begin{example} With Guadagni's tangle move, it is possible to construct a Lagrangian cobordism between the Legendrians pictured in Figure~\ref{fig:satellite}; see Figure~\ref{fig:guad-movie}.
However, at this point it is not known if Guadagni's tangle move is independent of the decomposable moves.    
\end{example}

\begin{figure}[!ht]
\centerline{\includegraphics[width=3.5in]{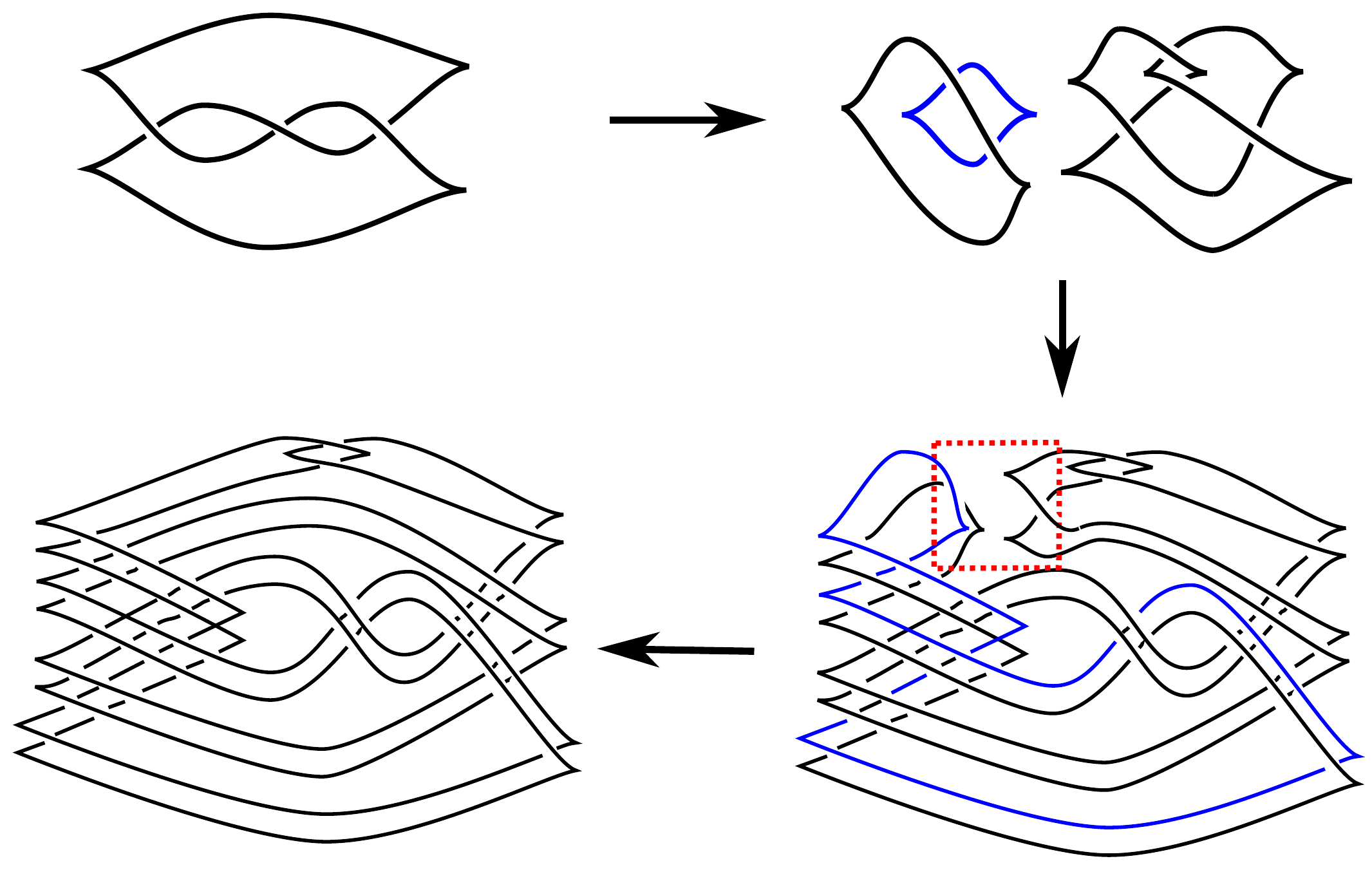}}
\caption{A movie, using a Guadagni move, of an (orientable, exact) Lagrangian cobordism from the trefoil  to the Whitehead double of  $m(9_{46})$ in Figure~\ref{fig:satellite}.}
\label{fig:guad-movie}
\end{figure}

\subsection{Lagrangian Diagram Moves} \label{ssec:lin}

As shown in Section~\ref{ssec:decomp-moves}, decomposable cobordisms are constructed from $0$-handles and  some $1$-handles (saddles) but no $2$-handles (caps).
Based on the work of Sauvaget \cite{Sau}, Lin \cite{Lin} constructs a genus two cap of a twice stabilized unknot, and thus gives the first explicit example of a non-decomposable Lagrangian cobordism. 
The construction describes time-slices of a Lagrangian cobordism through a list of moves on ``decorated Lagrangian diagrams.''

A {\bf decorated Lagrangian diagram} is a  curve  in the $xy$-plane with the compact regions  decorated by a positive number, which is the {\bf area} of the region.
Figure~\ref{fig:SLM} shows some examples: in the illustration of the $F$ move, $U$ is a Lagrangian projection of the Legendrian  unknot with maximal Thurston-Bennequin number;   in the illustration of the $C$ move, $U_{m}$  {\it is} a decorated Lagrangian diagram, but is  {\it not} the Lagrangian projection of a Legendrian knot.

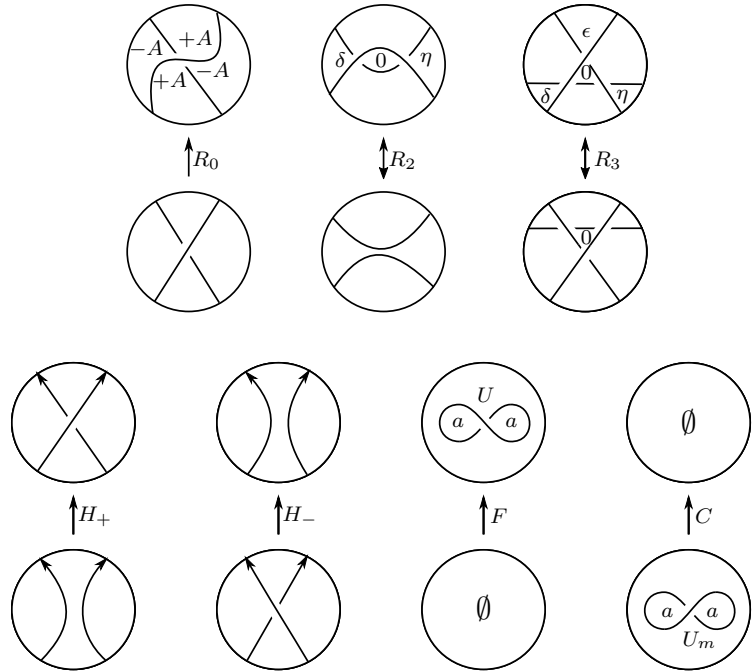
\begin{figure}[!ht]
\labellist
\pinlabel $\emptyset$ at 323 122
\pinlabel $\emptyset$ at 227 35
\tiny
\pinlabel $+A$ at 80 282
\pinlabel $+A$ at 92 302
\pinlabel $-A$ at 69 297
\pinlabel $-A$ at 100 287
\pinlabel $0$ at 179 292
\pinlabel $\delta$ at 160 292
\pinlabel $\eta$ at 200 292
\pinlabel $0$ at 275 286
\pinlabel $\e$ at  275 305
\pinlabel $\delta$ at 256 275
\pinlabel $\eta$ at 292 275
\pinlabel $0$ at 275 209
\pinlabel $H_+$ at 45 78
\pinlabel $H_-$ at 141 78
\pinlabel $F$ at 234 78
\pinlabel $C$ at 330 78
\pinlabel $R_0$ at 97 245
\pinlabel $R_2$ at 189 245
\pinlabel $R_3$ at  285 245
\pinlabel $U$ at  228 135
\pinlabel $a$ at  215 122
\pinlabel $a$ at 240 122
\pinlabel $a$ at  313 33 
\pinlabel $a$ at 334 33
\pinlabel $U_m$ at 328 20
\endlabellist
    \includegraphics[width=4in]{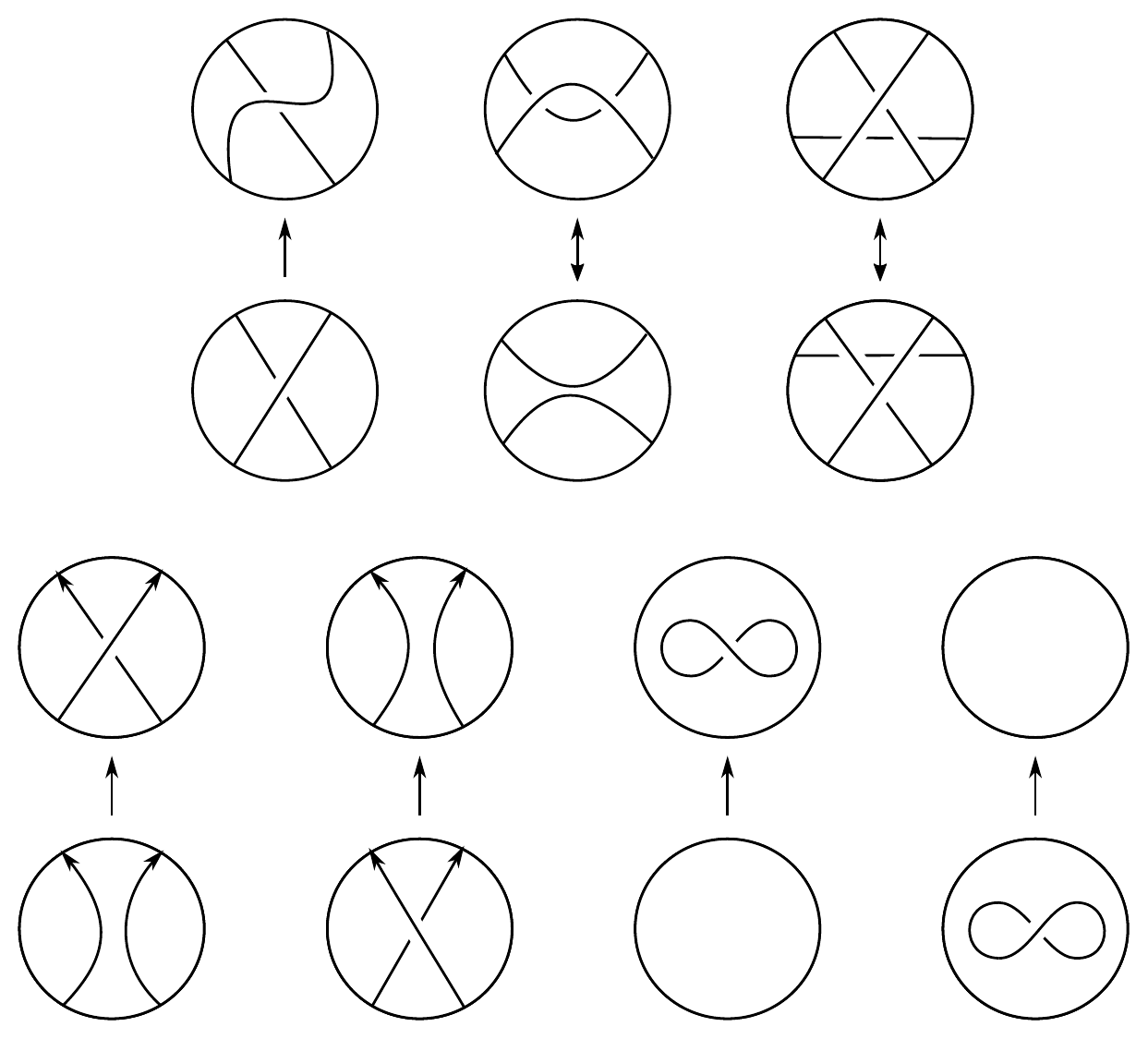}
    \caption{The Lagrangian diagram moves.  The labels in $R_0$ move represent the change of area through the move, while other labels  $0,\e, \delta, \eta, a$ indicate the area of the corresponding regions; here $0$ represents a positive area  that is smaller than either the area $\e$, the area  $\delta$ or  the area  $\eta$. 
   }
    \label{fig:SLM}
\end{figure}

\begin{theorem}[\cite{Lin}]
Let $\Lambda_{\pm}$ be Legendrian links and $D_{\pm}$ be their corresponding decorated Lagrangian projections. If one can create a sequence of decorated Lagrangian diagrams
$$D_-=D_0 \rightarrow D_1\rightarrow \cdots \rightarrow D_n=D_+$$ 
such that each diagram $D_{i+1}$ can be obtained from $D_i$ by the following combinatorial moves, then there is a compact Lagrangian submanifold in $\R \times \R^{3}$ with boundary $\Lambda_- \cup \Lambda_+$, where $\leg_{\pm} \subset \{\pm N\} \times \R^{3}$, for some $N > 0$.
\begin{enumerate}
\item $R_0$: a planar isotopy that changes areas by the amount $\pm A$, for $A > 0$.  This operation can only be done in the direction specified. 
%Regions with positive Reeb sign increase in area and regions with negative Reeb sign decrease in area. 
\item $R_2$: a Reidemeister II move. One can either introduce or eliminate two crossings assuming some area conditions are satisfied: it is possible to introduce or 
remove two crossings
as long as the area of the inner region, denoted by $0$ in the diagram, is less than either the area $\delta$ or the area $\eta$.
%of the region denoted by $\delta$ or by the area of the region denoted by  $\eta$.}    
One can also do this
move with the lower strand passing under the upper strand.
\item $R_3$: a Reidemeister III move.  One  can perform a Reidemeister III move as long as the area of the inner region, denoted by $0$ in the diagram, is less than either the area $\e$,  the area  $\delta$ or the area $\eta$.
%either of the areas of the regions denoted by $\delta$. 
The fixed center crossing can be reversed.   Additionally, the moving strand can also occur as an overstrand.
\item $H_+$: a handle attachment that creates a positive crossing in the diagram.
\item $H_-$: a handle attachment that removes a negative crossing in the diagram.
\item $F$: a filling that creates the diagram $U$, which is the Lagrangian projection of an unknot with maximal Thurston-Bennequin number.
\item $C$: a cap that eliminates the diagram $U_m$, which is the topological  mirror of $U$.
\end{enumerate}
These moves are called  {\bf Lagrangian diagram moves}.  Moreover, the constructed Lagrangian will be exact if, in addition,
\begin{enumerate}
\item [(E1)] Each move results in a diagram with all components having a total signed area equal to $0$.  The signed area of a region is determined by the sum of the signed heights of its Reeb chords.
\item[(E2)] If a handle attachment merges two components of a link,  the components being merged must be vertically split, meaning that the images of the $xy$-projections of these components are contained in
disjoint disks.  
 \end{enumerate}
\end{theorem}

\begin{remark}
\begin{enumerate}
\item For condition (E2), %moving in the increasing $t$-direction, 
the $H_{-}$ can never be applied to merge components, and $H_{+}$ can only be applied if the components being merged are vertically split. 
%Similarly, if starting at the ``top'' and working in the decreasing $t$-direction,  $H_{+}$ can never be applied and the $H_{-}$ move can only be applied to merge components that are vertically split.}

\item 
 A main distinction between the Lagrangian diagram moves and the decomposable moves is that each diagram $D_i$ in the middle of the sequence is not necessarily the Lagrangian projection of a Legendrian link. 
They are just the $xy$-projection of some time $t_i$-slice of the cobordism.
Thus the Lagrangian diagram moves are more flexible than the decomposable moves.   However, keeping track of the areas is an added complication.
\end{enumerate}
 \end{remark}

\begin{example} Figure~\ref{fig:torus-E1} illustrates the construction of a Lagrangian torus using the Lagrangian diagram moves. This torus fails to be exact since condition (E1) is violated.  Figure~\ref{fig:torus-E2} gives another construction of a Lagrangian torus.  This time,  all components have signed area $0$, but now condition (E2) is violated.

\end{example}

\begin{figure}[!ht]
\labellist
\pinlabel $\emptyset$ at -20 60
\pinlabel $\emptyset$ at 340 60
\pinlabel $F$ at 20 70
\pinlabel $H_-$ at  120 70
\pinlabel $H_+$ at  210 70
\pinlabel $C$ at  305 70
\pinlabel $a$  at 75 30
\pinlabel $a$ at 75 85
\pinlabel $a$  at 163 28
\pinlabel $a$ at 163 85
\pinlabel $a$  at 252 30
\pinlabel $a$ at 252 85
\endlabellist
\includegraphics[width=3in]{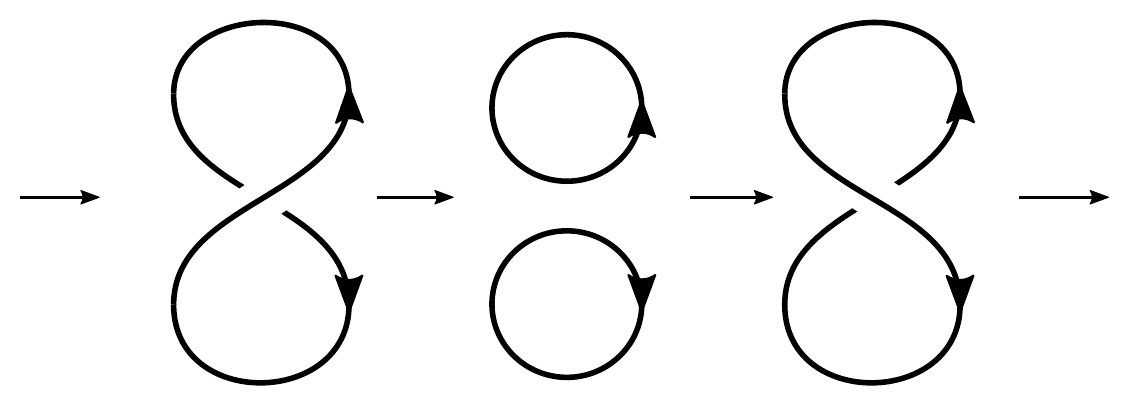}
\caption{A (non-exact) Lagrangian torus constructed using the Lagrangian diagram moves.  The middle figure violates (E1).}
\label{fig:torus-E1}
\end{figure}
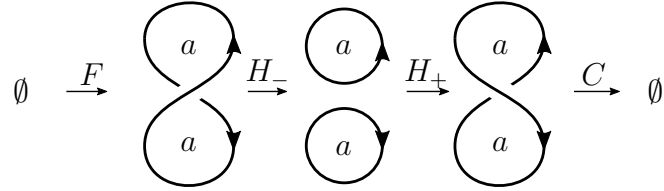

\begin{figure}[ht]
\labellist
\small
\pinlabel $\emptyset$ at 60 317
\pinlabel $6$ at 150 315
\pinlabel $6$ at 200 315

\pinlabel $6$ at 310 315
\pinlabel $5$ at 360 315
\pinlabel $0$ at 400 317
\pinlabel $1$ at 435 315

\pinlabel $6$ at 520 315
\pinlabel $5$ at 570 320
\pinlabel $3$ at 610 320
\pinlabel $4$ at 645 320

\pinlabel $6$ at 520 205
\pinlabel $5$ at 565 205
\pinlabel $1$ at  588 202
\pinlabel $3$ at 610 205
\pinlabel $4$ at 645 205

\pinlabel $6$ at 265 220
\pinlabel $0$ at 285 202
\pinlabel $1$ at 345 208
\pinlabel $5$ at  320 210
\pinlabel $3$ at 365 210
\pinlabel $4$ at 410 210
\pinlabel $0$ at 355 181

\pinlabel $6$ at 20 230
\pinlabel $1$ at 113 212
\pinlabel $5$ at  80 215
\pinlabel $3$ at 135 215
\pinlabel $4$ at 170 215
\pinlabel $0$ at  125 187

\pinlabel $2$ at 30 90
\pinlabel $1$ at 100 82
\pinlabel $5$ at  75 88
\pinlabel $3$ at 125 85
\pinlabel $0$ at 165 90
\pinlabel $4$ at  120 55

\pinlabel $5$ at 250 90
\pinlabel $5$ at 290 92
\pinlabel $1$  at 310 85
\pinlabel $4$ at 320 60

\pinlabel $9$ at 500 85
\pinlabel $9$ at 450 92
\pinlabel $1$  at 473 90
\pinlabel $0$ at 467 59

\pinlabel $9$ at 590 75
\pinlabel $9$ at 640 75
\pinlabel $\emptyset$ at  730 70

\endlabellist
\includegraphics[width=5in]{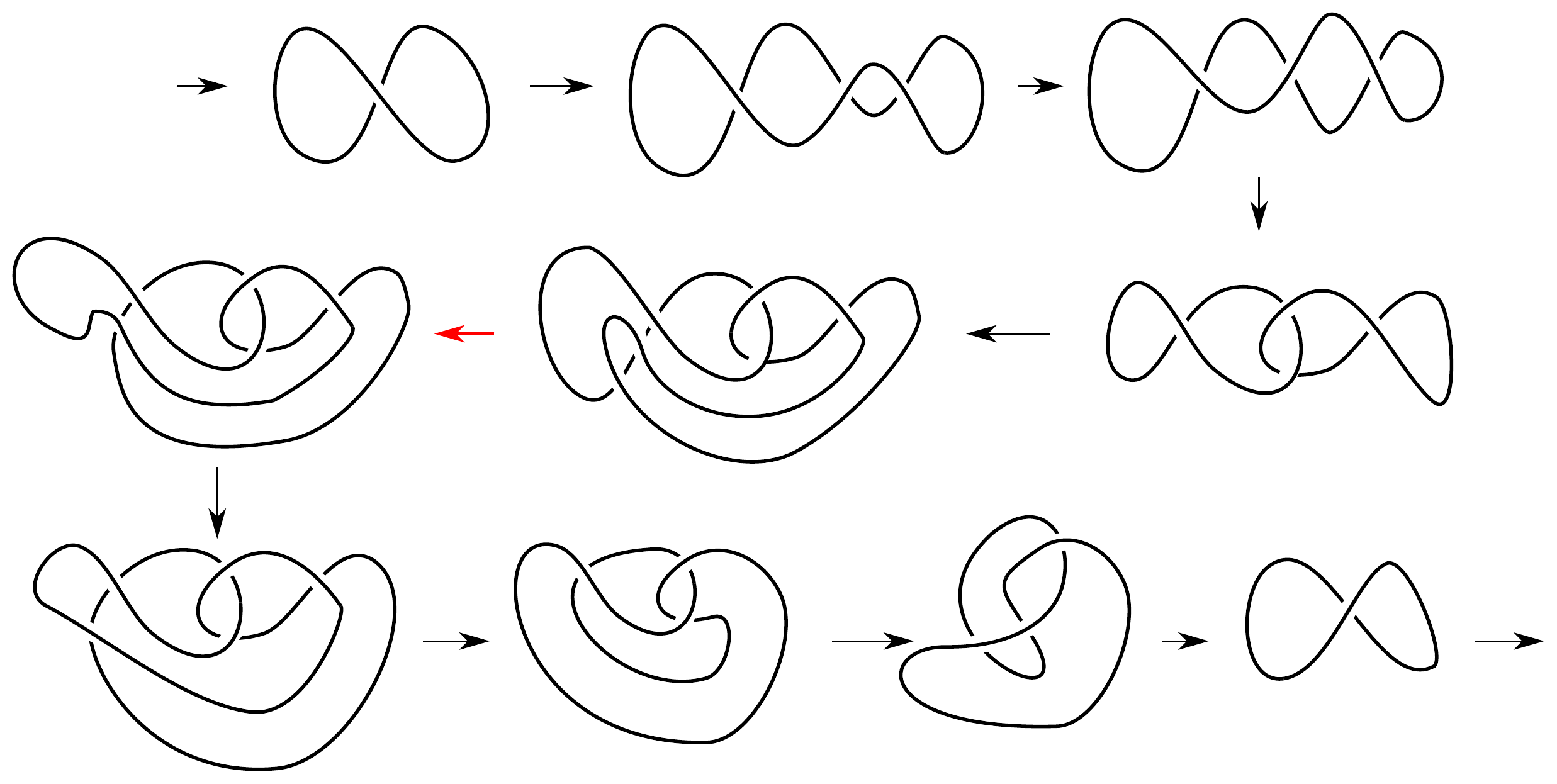}
\caption{A (non-exact) Lagrangian torus constructed using the Lagrangian diagram moves.  These figures satisfy (E1) but (E2) is violated in the step labelled by a red arrow.}
\label{fig:torus-E2}
\end{figure}

\section{Geometrical Constructions of Lagrangian Cobordisms} \label{sec:satellites}

An important general way to know of the existence of Lagrangian cobordisms without using the 
constructions described in Section~\ref{sec:construct} comes through the satellite operation.  
In this section, we review the satellite construction and then state results from \cite{CNS, GSY} about the existence of  a  Lagrangian concordance/cobordism from $\leg_{-}$ to $\leg_{+}$ implying the existence of
a Lagrangian concordance/cobordism between corresponding satellites.

\subsection{The Legendrian Satellite Construction}\label{sec:sat}
 We begin by reviewing the construction of a Legendrian satellite; see also   \cite[Appendix]{NgTra} and \cite[Section 2.2]{CNS}.
To construct a Legendrian satellite, begin by identifying the open solid torus $S^{1} \times \R^{2}$ with the $1$-jet space of the circle, $J^{1}S^{1} \cong T^{*}S^{1} \times \R$, equipped with the
contact form $\alpha = dz - ydx$, where $x,y$ are the coordinates in $T^{*}S^{1}$ and $z$ is the coordinate in $\R$.  Similar to the situation for $\R^{3} \cong J^{1}\R$, we can recover a Legendrian
knot in $J^{1}S^{1}$ from its front projection in $S^{1}_{x} \times \R_{z}$, which is typically drawn by representing $S^{1}$ as an interval with its endpoints identified.  

Given an oriented Legendrian {\bf companion} knot $\Lambda \subset \R^3$ and a oriented Legendrian {\bf pattern} knot $P \subset J^1(S^1)$, the Legendrian
 neighborhood theorem says that $\leg$ has a standard neighborhood $N(\leg)$
such that there is a contactomorphism $\kappa: J^{1}(S^{1}) \to N(\leg)$.  The {\bf Legendrian satellite}, $S(\Lambda, P)$, is then the image $\kappa(P)$.  
 The  front projection of  $S(\Lambda, P)$ is as shown in Figure~\ref{fig:sat}.
In particular, suppose that the front projection of the
 pattern $P$ intersects the vertical line at the  boundary  of the $S^{1}$ interval
  $n$ times. We then make an $n$-copy of $\Lambda$ by using  $n$-disjoint copies of $\Lambda$ that all differ by small translations in the $z$-direction.
 Take a point  on the front projection of $\Lambda$ that is oriented from left to right, cut the front of the $n$-copy open along the $n$-copy at that point, and insert the front diagram of $P$.
 The orientation on the satellite $S(\Lambda, P)$ is induced by the orientation on $P$.

\begin{figure}[!ht]
\labellist
\small
\pinlabel $\Lambda$ at -10 130
\pinlabel $P$ at 0 20
\pinlabel $S(\Lambda,P)$ at  290 -5
\endlabellist
\includegraphics[width=3in]{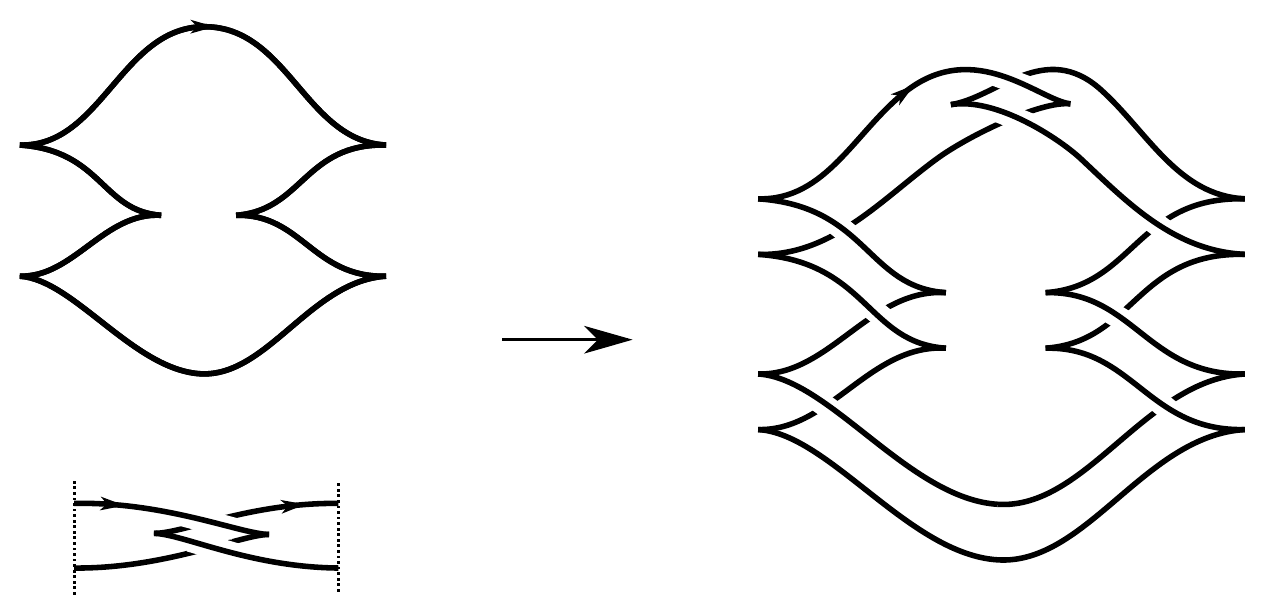}
\caption{A example of  Legendrian satellite.}
\label{fig:sat}
\end{figure}

\begin{remark}\label{rem:sat-nicer}
The satellite operation often makes Legendrian knots ``nicer''; for example, in Figure~\ref{fig:sat}, the companion $\leg$ is stabilized and does not admit an augmentation or a normal  ruling.
However, the satellite $S(\leg, P)$ does admit a normal  ruling and augmentation.
\end{remark}

\subsection{Lagrangian Cobordisms for Satellites}

In \cite[Theorem 2.4]{CNS}, Cornwell, Ng, and Sivek, show that Lagrangian {\it concordance} is preserved by the Legendrian
satellite operation.   

\begin{theorem}[{\cite{CNS}}] \label{thm:satellite} Suppose $P \subset J^{1}S^{1}$ is a Legendrian knot.  If there exists a Lagrangian
concordance $L$ from a Lengendrian knot $\leg_{-}$ to a Lengendrian knot $\leg_{+}$, then there exists a Lagrangian concordance $L_{P}$ from $S(\leg_{-}, P)$
to $S(\leg_{+}, P)$.
\end{theorem}

In particular, as shown in Figure~\ref{fig:m946}, there is a Lagrangian concordance from $\leg_{-}$, which is the Legendrian unknot with $tb = -1$,
to  $\leg_{+}$, which is the Legendrian  $m(9_{46})$ with maximal $tb = -1$.  Using the Legendrian ``clasp'' tangle $P$ as shown in Figure~\ref{fig:sat} -- which produces the
Legendrian Whitehead double -- we can conclude that there exists a Lagrangian concordance 
from $S(\leg_{-}, P)$ to $S(\leg_{+}, P)$.  In fact, $S(\leg_{-}, P)$ is the positive trefoil with $tb = 1$.  Thus Theorem~\ref{thm:satellite} implies that 
  there exists a Lagrangian concordance
between the Legendrian knots in Figure~\ref{fig:satellite}.
 
%It is conjectured, \cite[Conjecture 3.3]{CNS}: 
\begin{conjecture}[{\cite[Conjecture 3.3]{CNS}}]  \label{conj:cns}
 The Lagrangian concordance from $S(\leg_{-}, P)$ to $S(\leg_{+}, P)$ built through the satellite construction is not decomposable. 
  \end{conjecture}

Theorem~\ref{thm:satellite} has been extended to higher genus cobordisms by Guadagni, Sabloff, and Yacavone in \cite{GSY}.   
%To state their theorem, we need to first introduce the notion of ``twisting'' a pattern $P$.  Given a pattern $P \subset J^{1}S^{1}$, $\Delta P$ is the pattern obtained by the addition of a full twist as illustrated in Figure~\ref{fig:full-twist}; $\Delta^{t} P$ can be thought of as $P$ followed by $t$ full twists.
To state their theorem,
 we need to first introduce the notion of ``twisting'' and then closing a tangle $T \subset J^1[0,1]$.  Given a Legendrian tangle $T \subset J^{1}[0,1]$, $\Delta T$
  is the tangle obtained by adding the tangle $T$ and the full twist tangle $\Delta$, which 
is illustrated in Figure~\ref{fig:full-twist}; the tangle $\Delta^{t} T$ can be thought of as $T$ followed by $t$ full twists.  Given a Legendrian tangle $T \subset J^{1}[0,1]$, $\overline{T} \subset J^1(S^1)$ will
 denote the associated closure to a Legendrian link.

\begin{figure}[!ht]
\labellist
\pinlabel $\vdots$ at 150 100
\endlabellist
\includegraphics[width=1in]{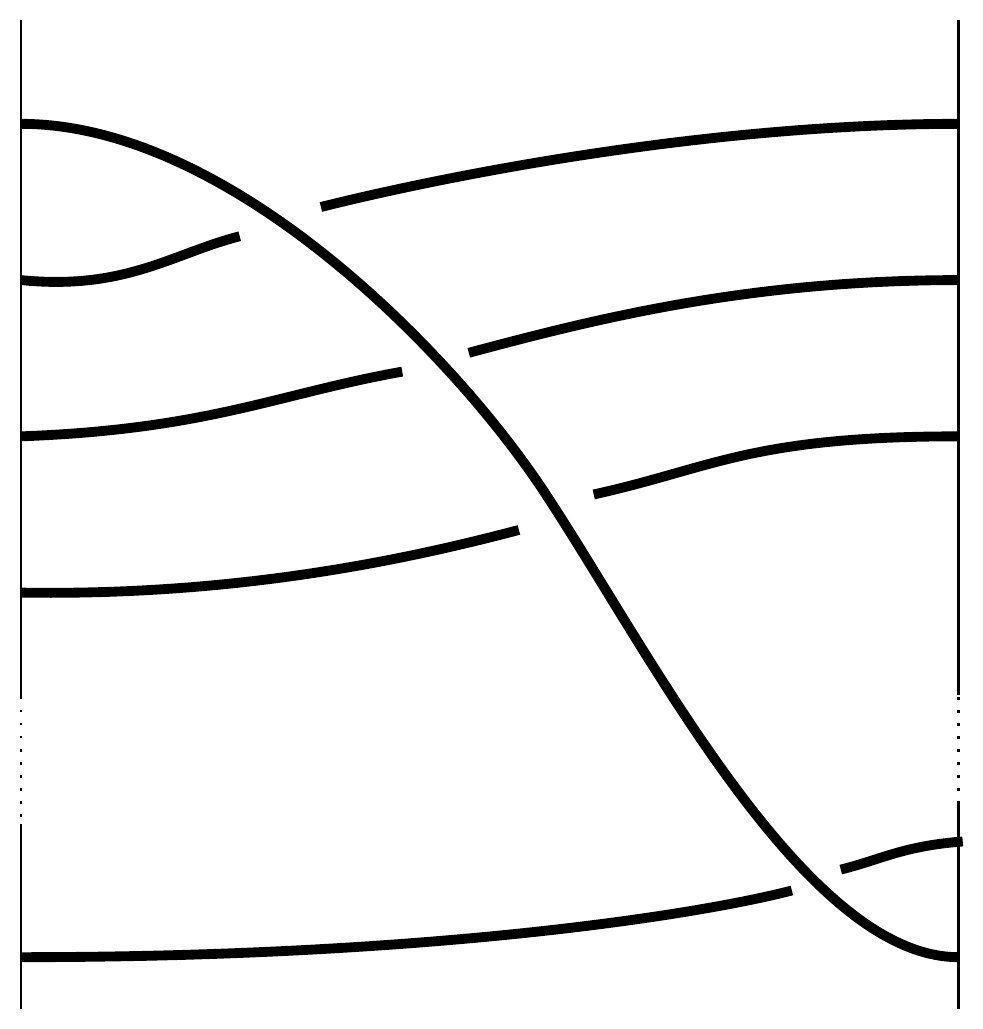}
\caption{For an $n$-stranded tangle, repeating this basic tangle $n$ times produces a full twist. }
\label{fig:full-twist}
\end{figure}

%\begin{theorem}[\cite{GSY}] \label{thm:sat-genus}
%Suppose $P \subset J^{1}S^{1}$ is a Legendrian knot.  If there exists a  Lagrangian cobordism $L$ from  $\leg_{-}$ to $\leg_{+}$ of genus $g(L)$, then there exists a Lagrangian cobordism $L_{P}$ from $S(\leg_{-}, \Delta^{2g(L) + 1} P)$ to $S(\leg_{+}, \Delta P)$.\end{theorem}

\begin{theorem}[\cite{GSY}] \label{thm:sat-genus}
Suppose $T \subset J^{1}[0,1]$ is a Legendrian tangle whose closure $\overline{T} \subset J^{1}(S^{1})$ is a Legendrian knot.  If there exists a  Lagrangian
cobordism $L$ from  $\leg_{-}$ to $\leg_{+}$ of genus $g(L)$, then there exists a Lagrangian cobordism $L_{T}$ from
$S(\leg_{-}, \overline{\Delta^{2g(L) + 1} T})$ to $S(\leg_{+}, \overline{\Delta T})$.
\end{theorem}

 In fact, Theorem~\ref{thm:sat-genus} can be generalized to use the closure of different tangles $T_{-}$ and $T_{+}$  that are Lagrangian cobordant; for details, see \cite{GSY}.

%In fact, Theorem~\ref{thm:sat-genus} can be generalized to use different patterns $P_{-}$ and $P_{+}$ that are Lagrangian cobordant; for details, see \cite{GSY}.

\begin{remark}
It is natural to wonder if,  along the lines of Conjecture~\ref{conj:cns},  this higher genus satellite procedure 
can create additional {\it candidates} for Legendrians that can be connected by a Lagrangian cobordism but not by a 
 decomposable Lagrangian cobordism.   In \cite[Theorem 1.5]{GSY},  it is shown that 
if the cobordism $L$ from $\leg_{-}$ to $\leg_{+}$ is decomposable and the handles in the decomposition satisfy conditions known
as ``Property A'', then the corresponding satellites $S(\leg_{-}, \Delta^{2g(L) + 1} P)$ and $S(\leg_{+}, \Delta P)$ will also
be connected by a decomposable Lagrangian cobordism.  In particular,  if there exists a decomposable cobordism $L$ that does {\it not} satisfy
Property A and is not isotopic to a cobordism that satisfies Property A, then the satellite construction would lead to a higher
genus candidate that generalizes Conjecture~\ref{conj:cns}. 
\end{remark}

\subsection{Obstructions to Cobordisms through Satellites} \label{ssec:sat-obstruct}
 In Section~\ref{ssec:obstructions}, some known obstructions to the existence of a Lagrangian cobordism were mentioned.  As mentioned in Remark~\ref{rem:no-stab-obstruct},
  a  number of these obstructions
 require $\leg_{-}$ to admit an augmentation, and thus in particular $\leg_{-}$ must be non-stabilized. However, as mentioned in Remark~\ref{rem:sat-nicer}, it is possible for 
 the satellite of a Legendrian $\leg$ to admit an augmentation even if $\leg$ does not.  So
the contrapositive of Theorem~\ref{thm:satellite} provides a potential strategy for further obstructions to the existence of a Lagrangian cobordism when $\leg_{-}$ does not
admit an augmentation.  For example,  motivated by Obstruction (4) in Section~\ref{ssec:obstructions}, one can ask:
{\it Can a count of augmentations give an obstruction to the existence of a Lagrangian  concordance from $S(\leg_{-}, P)$ to
$S(\leg_{+}, P)$ and thereby obstruct the existence of a Lagrangian concordance from $\leg_{-}$ to $\leg_{+}$?}
In fact, this augmentation count will not likely provide a further obstruction:
a simple computation shows that when $\leg$ is stabilized enough,  the number of augmentations  of $S(\Lambda, P)$ only depends on the Legendrian pattern  $P$.
If trying to pursue this path to obtain further obstructions to Lagrangian cobordisms, it is useful to keep in mind   the following result of Ng that shows  the DGA of the satellite of a Legendrian $\leg$ might only remember
the underlying knot type of $\leg$.
\begin{theorem}[\cite{Nthesis}]
Suppose $\Lambda_1$ and $\Lambda_2$ are stabilized Legendrian knots that are of the same topological knot type and have the same Thurston-Bennequin and rotation numbers.
For a Legendrian pattern $P$ whose front intersects a vertical line by two points, the DGAs  of $S(\Lambda_1, P)$  and $S(\Lambda_2, P)$ are equivalent.
\end{theorem}

\section{Candidates for Non-Decomposable Lagrangian Cobordisms} \label{sec:non-decomp-candidates}

Now that we have developed some ways to construct a Lagrangian cobordism through combinatorial moves and satellites, we state some theorems that show
{\it if} a Lagrangian cobordism does exist, then 
it cannot be decomposable: this addresses Motivating Question~(\ref{ques:decomp}).  
While we discuss these theorems, it is useful to keep in mind the known  obstructions to Lagrangian cobordisms that were mentioned in 
 Section~\ref{ssec:obstructions}.

\subsection{Candidates for Non-decomposable Lagrangian Cobordisms from Normal Rulings}\label{sec:rul} One simple way to show that two Legendrians $\leg_\pm$ cannot be
connected by a decomposable Lagrangian cobordism comes from a count of ``combinatorial'' rulings.
 Roughly,  
a {\bf normal ruling} of a Legendrian $\Lambda$ is a ``decomposition'' of the front projection into pairs of paths from left cusps to right cusps such that 
\begin{enumerate}
\item each pair of paths starts from a common left cusp and ends at a common right cusp, has no further intersections, and bounds a topological disk whose boundary is smooth everywhere other than at the cusps and certain crossings  called {\bf switches}, and
\item near a switch, the pair of paths must be arranged as in one of the diagrams in Figure~\ref{fig:rul}; observe that near the switch,  vertical slices of the associated disks are either disjoint or the slices of one are contained in the slices of the other.
\end{enumerate}
Formal definitions of normal  rulings can be found in, for example, \cite{CP} and \cite{Fuchs}. 
\begin{figure}[!ht]
\includegraphics[width=4in]{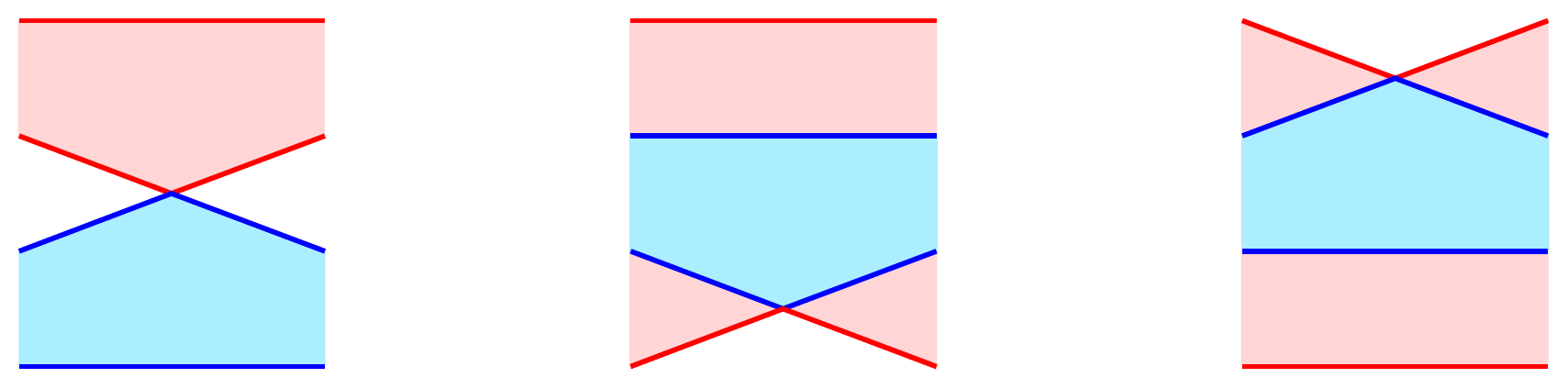}
\caption{Normal rulings near a switch.}
\label{fig:rul}
\end{figure}

As an illustration, all normal rulings of a particular Legendrian trefoil are shown in Figure~\ref{fig:trefrul}.

\begin{figure}[!ht]
\includegraphics[width=4in]{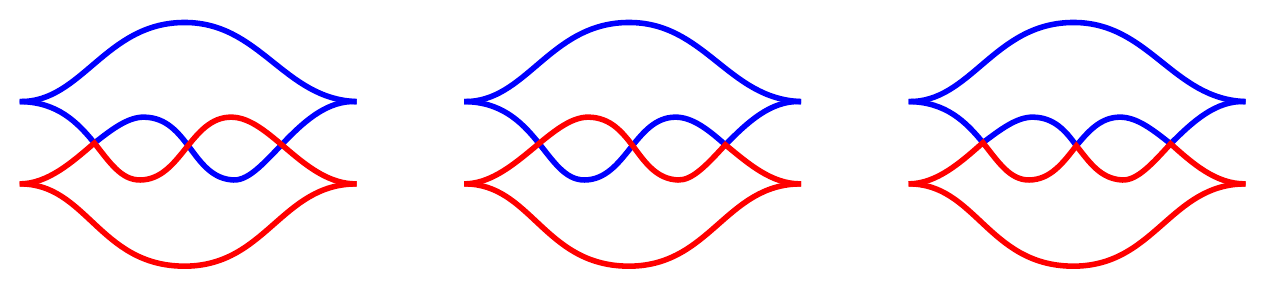}
\caption{All normal rulings of this max $tb$ positive Legendrian trefoil.}
\label{fig:trefrul}
\end{figure}

For each normal  ruling $R$, let $s(R)$ and $d(R)$ be the number of switches and number of disks, respectively. 
By \cite{CP}, the {\bf ruling polynomial} is
$$R_{\Lambda}(z)= \sum_{R} z^{s(R)-d(R)},$$ where the sum is over all the normal rulings,  is an invariant of $\Lambda$ under Legendrian isotopy. Normal rulings and augmentations are closely related even though they are defined in very different ways \cite{Fuchs, FI, NS, Sa}. 
 
We have the following obstruction to decomposable cobordisms in terms of  normal rulings.
\begin{theorem} If $\leg_{-}$ has $m$ normal rulings and $\leg_{+}$ has $n$ normal rulings with $m > n$,  then there is no decomposable
Lagrangian cobordism from $\leg_{-}$ to $\leg_{+}$.
\end{theorem}

\begin{proof} 
One can compare the number of normal  rulings of the two ends for the decomposable moves, as shown in Figure~\ref{fig:ruling}. 
Thus any normal  ruling of $\Lambda_-$ induces a normal  ruling of $\Lambda_+$. 
Different normal rulings of $\Lambda_-$ induce different normal rulings of $\Lambda_+$.
Therefore the number of normal  rulings of $\Lambda_+$ is bigger than or equal to the number of normal  rulings of $\Lambda_-$.
\begin{figure}[!ht]
\includegraphics[width=2in]{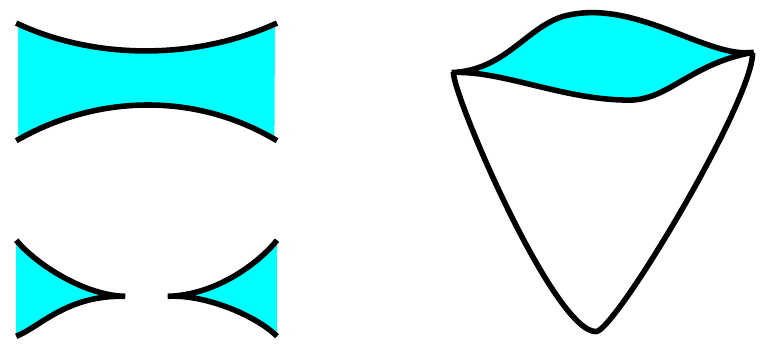}
\caption{Comparison of normal  rulings for decomposable moves.}
\label{fig:ruling}
\end{figure}
\end{proof}

Here is a strategy to  show the 
existence of Legendrians that can be connected by a Lagrangian cobordism but not
by one that is decomposable.

\begin{strat}    \label{strat:rulings} Choose Legendrians $\leg_{\pm}$ such that:
\begin{enumerate}
\item$\leg_{+}$ has fewer graded normal rulings than $\leg_{-}$, and
\item it is possible to construct, via a combination of the combinatorial constructions from Section~\ref{sec:construct}  
or the satellite construction from Section~\ref{sec:satellites}, a Lagrangian cobordism from
$\leg_{-}$ to $\leg_{+}$.   
\end{enumerate}
\end{strat}

\begin{remark} If $\leg_{\pm}$ admit normal rulings, they will admit augmentations \cite{FI, Sa}.
From Section~\ref{ssec:obstructions} obstructions (4)b, we then know that if there is a Lagrangian cobordism from $\Lambda_-$ to $\Lambda_+$, their ruling polynomials satisfy 
$$R_{\Lambda_-}(q^{1/2}-q^{-1/2})\leq q^{-\chi(\Sigma)/2}R_{\Lambda_+}(q^{1/2}-q^{-1/2}),$$  for any $q$ that is a  power of a prime number.
Satisfying condition (1) in Strategy~\ref{strat:rulings} means that the polynomial on the  right side of the inequality has fewer terms than the polynomial on the left side of the inequality.
 If following this approach, it may be helpful to start by  first finding a pair of  positive integer coefficient polynomials that satisfy this inequality and condition (1) at the same time. 
 One can start with checking the ruling polynomials of small crossing number Legendrian knots on \cite{ChNg}.
\end{remark}

\subsection{Candidates for Non-decomposable Lagrangian Concordances from Topology}  

Observe that any decomposable Lagrangian concordance will be a smooth ribbon concordance.  Thus it is potentially possible to use known
obstructions to ribbon concordances to find examples of smooth knots whose Legendrian representatives cannot be connected by a decomposable Lagrangian concordance: 
constructing a Lagrangian concordance between very stabilized Legendrian representatives of these knot types, via the combinatorial techniques of Section~\ref{sec:construct} or geometric techniques of
 Section~\ref{sec:satellites},
 will give an example of an exact Lagrangian concordance between knots that cannot be connected by a decomposable Lagrangian concordance.

For example, it is known \cite{Gor, Zem, LZ} that the only knot that admits a ribbon 
concordance to the unknot is the unknot itself.  This has as a corollary the following
obstruction to a decomposable Lagrangian concordance.

\begin{theorem}[{\cite[Theorem 3.2]{CNS}}]  \label{thm:cns}  If $\leg_-$ is topologically non-trivial and $\leg_{+}$ is topologically an unknot, then
there is no decomposable Lagrangian concordance  from $\leg_{-}$ to $\leg_{+}$.
\end{theorem}

\begin{example}
To illustrate this theorem, here is a possible low crossing number Legendrian knot to examine as $\leg_{-}$. Consider the topological knot $6_1$ which is slice and ribbon. Its maximum $tb$ Legendrian representative $\Lambda_{6_1}$ (see Figure~\ref{fig:6_1}) has $tb=-5$ and $r=0$. 
The DGA of this Legendrian $\alg(\Lambda_{6_1})$ admits an augmentation, and thus  $\leg_{6_{1}}$ does not admit a Lagrangian cap; see obstructions (6) in Section~\ref{ssec:obstructions}.  
Since we are trying to construct
a Legendrian $\leg_-$ that could be Lagrangian concordant to a stabilized unknot, which might have a Lagrangian cap, we will add some stabilizations that will prevent augmentations and thereby allow the
possibility of a Lagrangian cap.
If we now  add a positive and a negative stabilization to $\Lambda_{6_1}$, we get a knot $\Lambda_{6_1}^\pm$ with $tb=-7$ and $r=0$, which has no augmentation and is still topologically the
knot $6_1$. If, by a sequence of moves in Section~\ref{sec:construct}, one can construct a concordance from $\Lambda_{6_1}^\pm$ to the $tb=-7$ stabilized unknot, then by Theorem ~\ref{thm:cns} this
Lagrangian concordance will not be decomposable; see Figure~\ref{fig:stab6_1}.  In fact, one can stabilize $\Lambda_{6_1}$ {\it as many times as we wish} resulting in $tb(\leg_{-}) = t$ and $r(\leg_{-}) = r$ and try, using the combinatorial constructions of Section~\ref{sec:construct}, to construct a Lagrangian concordance to $\leg_{+}$, where  $\leg_{+}$  is a 
Legendrian unknot with $tb(\leg_{+}) = t$ and $r(\leg_{+}) = r$.  {\it If possible}, such a construction would prove the existence of a non-decomposable Lagrangian concordance.
\end{example}
\begin{figure}[!ht]
	\includegraphics[width=2in]{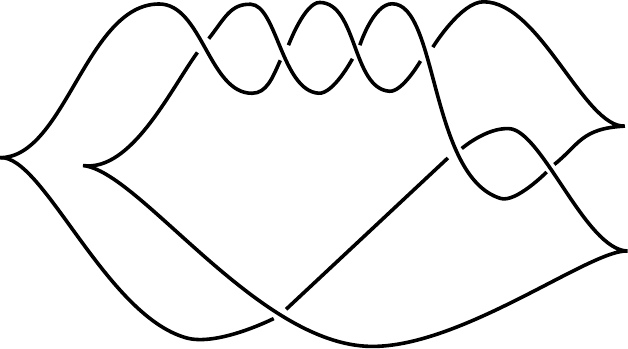}
	\caption{Front diagram of $\Lambda_{6_1}$.}
	\label{fig:6_1}
\end{figure}
\begin{figure}[!ht]
	\includegraphics[width=3in]{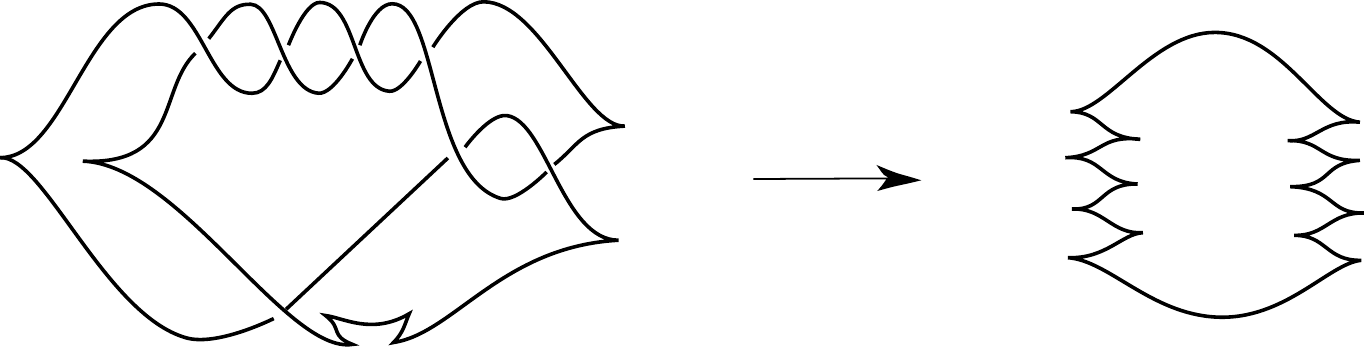}
	\caption{Any Lagrangian concordance from the doubly stabilized $\Lambda_{6_1}$ to the  $tb=-7$, $r = 0$ Legendrian unknot would necessarily be
	non-decomposable.}
	\label{fig:stab6_1}
\end{figure}

There are additional results from topology that give obstructions to the existence of ribbon concordances.  For example, 
as shown by Gilmer \cite{Gil} and generalized by Friedl and Powell \cite{FP}, if $K_-$ is ribbon concordant to $K_+$, then the Alexander polynomial of $K_-$ divides the Alexander polynomial of $K_+$.  
We can invoke these results in a strategy to show the existence of non-decomposable Lagrangian concordances.
\begin{strat}  \label{strat:not-ribbon}
\begin{enumerate}
\item Use results from smooth topology  to find examples of smooth knots $K_{\pm}$ such that $K_-$ is not ribbon concordant to $K_+$. 
\item For {\it any} pair of Legendrian representatives $\leg_{\pm}$ of the knot type $K_{\pm}$, even highly stabilized, use a combination of the combinatorial moves
described in Section~\ref{sec:construct}   
to construct a Lagrangian concordance from $\leg_{-}$ to $\leg_{+}$.
 \end{enumerate}
\end{strat}
The example with the knot $6_1$ given above is a concrete example to try to apply this strategy with $K_-=6_1$ and $K_+$  being an unknot. A possible example when $K_+$ is non-trivial is the following.
\begin{example}
Let $K_{-}$ be the connect sum of the right- and left-handed trefoils, $K_-=T_r\#T_l$, and let $K_{+}$ be the connect sum of the figure 8 knot with itself, $K_+=F_8\#F_8$. 
These knots are concordant but there is no ribbon concordance from $K_-$ to $K_+$, as first shown by Gordon \cite{Gor}.  Choose Legendrian representatives $\leg_{\pm}$ of $K_{\pm}$ such that 
$tb(\leg_{-}) = tb(\leg_{+})$ and $r(\leg_{-}) = r(\leg_{+})$; note that $\leg_{\pm}$ can be very stabilized.  {\it If} we can construct a Lagrangian concordance from $\leg_{-}$ to $\leg_{+}$, via 
the combinatorial moves of Section~\ref{sec:construct}, then we will have shown the existence of a pair of Legendrians that are (exactly, orientably) Lagrangian concordant but cannot be connected by a
decomposable Lagrangian concordance.  
\end{example}

\begin{remark} Some known obstructions to ribbon concordance  are, in fact, obstructions to generalizations of ribbon concordance, namely {\bf strong homotopy ribbon concordance} and {\bf homotopy ribbon concordance}. A strong homotopy ribbon concordance is one whose complement is ribbon, i.e., can be built with only 1-handles and 2-handles. A homotopy ribbon concordance from $K_-$ to $K_+$ is a concordance where the induced map on $\pi_1$ of the complement of $K_-$ (resp. $K_+$) injects (resp. surjects) into $\pi_1$ of the complement of the concordance. Gordon \cite{Gor} showed that 
$$\begin{aligned}
\text{ribbon concordant } & \implies \text{ strong homotopy ribbon concordant} \\ &\implies \text{ homotopy ribbon concordant. }
\end{aligned}
$$
There have been a number of recent results  obstructing (homotopy or strong homotopy) ribbon concordances from Heegaard-Floer and Khovanov homology \cite{Zem, LZ, MZ, GL};
these results play an important role in Strategy~\ref{strat:not-ribbon}.
\end{remark}

\subsection{Candidates for Non-decomposable Lagrangian Cobordisms from GRID Invariants} \label{ssec:Floer}
Some candidates for non-decomposable Lagrangian cobordisms of higher genus come from knot Floer homology.
Using the grid  formulation of knot Floer homology \cite{OST}, 
Ozsv\'ath, Szab\'o, and Thurston defined Legendrian invariants of a Legendrian link  $\leg \subset \R^3$, called GRID invariants,
which are  elements
  in the {\it hat} flavor of knot Floer homology of $\leg \subset -S^3$:
$$\widehat\lambda^+(\leg), \widehat\lambda^-(\leg) \in \widehat{HFK}(-S^3, \leg).$$
For more background, see \cite{OST, MOS}.

Baldwin, Lidman, and Wong \cite{BLW} have shown that these GRID invariants can be used to obstruct the existence of decomposable Lagrangian
cobordisms.
\begin{theorem}[{\cite[Theorem 1.2]{BLW} \label{thm:blw}}]
 Suppose that $\leg_{\pm}$ are Legendrian links in $\R^{3}$ such that either
\begin{enumerate}
\item $\widehat\lambda^+(\leg_{+}) = 0$ and $\widehat\lambda^+(\leg_{-}) \neq 0$, or
\item $\widehat\lambda^-(\leg_{+}) = 0$ and $\widehat\lambda^-(\leg_{-}) \neq 0$. 
\end{enumerate}
Then there is no decomposable Lagrangian cobordism from $\leg_{-}$ to $\leg_{+}$.
\end{theorem}

\begin{remark} 
By \cite{BVV}, in the standard contact manifold $\R^3$, the GRID invariants agree with the LOSS invariant \cite{LOSS}.
The LOSS invariant is functorial on Lagrangian concordances by \cite{BSa, BS}.
Thus Theorem~\ref{thm:blw} would also obstruct the existence of general Lagrangian concordances and not only the decomposable ones.
To find non-decomposable cobordisms using obstructions from \cite{BLW}, we should focus on  non-zero genus cobordisms. 
\end{remark}

Using the facts that the GRID invariants are non-zero for the $tb = -1$ Legendrian unknot 
and that $\widehat\lambda^+(\leg_{+})$ (resp. $\widehat\lambda^-(\leg_{+})$) vanish for positively (negatively) stabilized Legendrian links,
Theorem~\ref{thm:blw} gives the following corollary.
\begin{corollary}[{\cite[Corollaries 1.3, 1.4]{BLW}}] \label{cor:BLW}
\begin{enumerate}
\item If $\leg \subset \R^{3}$ is a Legendrian link such that $\widehat\lambda^+(\leg) = 0$ or $\widehat\lambda^-(\leg) = 0$, then there is no 
decomposable Lagrangian filling of $\leg$.
\item Suppose $\leg_{\pm}$ are Legendrian links such that either
\begin{enumerate}
\item  $\widehat\lambda^+(\leg_{-}) \neq 0$ and $\leg_{+}$ is the positive stabilization of a Legendrian link, or
\item  $\widehat\lambda^-(\leg_{-}) \neq 0$ and $\leg_{+}$ is the negative stabilization of a Legendrian link.
\end{enumerate}
Then there is no decomposable Lagrangian cobordism from $\leg_{-}$ to $\leg_{+}$.

\end{enumerate}
\end{corollary}

This provides another strategy to show the existence of Legendrians $\leg_{\pm}$ that are  Lagrangian cobordant
but cannot be connected by a decomposable Lagrangian cobordism.
 \begin{strat}\label{strat:blw}
 \begin{enumerate}
\item Find Legendrians $\leg_{\pm}$ satisfying the GRID invariants conditions of Corollary~\ref{cor:BLW} and Theorem~\ref{thm:blw}  such that there are no known obstructions, as described in Section~\ref{ssec:obstructions}, to the existence
of a Lagrangian cobordism from $\leg_{-}$ to $\leg_{+}$.
\item Use a combination of the combinatorial moves described in Section~\ref{sec:construct} to construct a %(exact, orientable) 
Lagrangian cobordism  from $\leg_{-}$ to $\leg_{+}$. 
\end{enumerate}
\end{strat}
\begin{example}
Concrete examples mentioned in \cite[Section 4.1]{BLW}  can be used for Strategy~\ref{strat:blw}.  
Let $\Lambda_0, \Lambda_1$ be the Legendrian $m(10_{132})$ knots and  Legendrian $m(12{n_{200}})$ knots shown in  \cite[Figures 2 and 3]{NOT}.  Modify them with a pattern shown in \cite[Figure 13]{BLW} to get $\Lambda'_0$ and $\Lambda'_1$, which are of  knot type $m(12n_{199})$ and $m(14n_{5047})$ (or its mirror), respectively.
%From Tye:  - p.23, Example 19.  12n5047 should be 14n5047.  (One small very recent caveat about this: 
%Mike did the relevant computations with this knot in Knotscape.  According to Chuck Livingston, for knots with at least 13 crossings, 
%Knotscape doesn't distinguish between knots and their mirrors in its labeling, since it works with DT codes.  So, my understanding is that it turns out we did not actually check if our knot is 14n5047 or its mirror.  
%Mike did all the computer work on our paper so you could reach out to him for a bit more info.  Sorry for the suboptimal explanation, I just found out about this recently and Mike is talking to the journal about it.)  
For $i,j=0,1$ we have $tb(\Lambda'_i)=tb(\Lambda_i)+2$ and $r(\Lambda'_i)=r(\Lambda_i)$. There is no decomposable
Lagrangian cobordism from
\begin{enumerate}
\item  $\Lambda_0$ to $\Lambda'_1$, or
\item  $\Lambda_1$ to $\Lambda'_0$.
\end{enumerate}
{\it If} we can construct, using the combinatorial techniques of Section~\ref{sec:construct}, a Lagrangian cobordism (necessarily of genus $1$) from $\Lambda_{0}$ to $\Lambda'_{1}$ or from $\Lambda_1$ to $\Lambda'_0$, then we will have found a
non-decomposable Lagrangian cobordisms. 
\end{example}

\begin{example} 
In \cite[Section 4.3]{BLW}, the authors provide an infinite family of pairs of Legendrian knots where there does not exist a decomposable Lagrangian cobordism between them.
\end{example}
 
\begin{remark} In Strategies~\ref{strat:not-ribbon} and~\ref{strat:blw}, we emphasized the construction of Lagrangian cobordisms using the 
combinatorial techniques of Section~\ref{sec:construct}.  It would be interesting to know if the geometric constructions of Section~\ref{sec:satellites} could also
be used to show the existence of a Lagrangian concordance/cobordism  from the theory of normal  rulings, topology, or grid invariants, that are known to not  be 
decomposable.
\end{remark}

\subsection{Non-decomposable Candidates through Surgery}
 An additional strategy to show the existence of a non-decomposable Lagrangian filling comes from understanding properties
of the contact manifold that is obtained from surgery on the Legendrian knot.  In particular, Conway, Etnyre, and Tosun \cite{CET}
have detected a relationship between Lagrangian fillings of a Legendrian and symplectic fillings of the contact manifold
obtained by performing a particular type of surgery on the Legendrian.

\begin{theorem}[{\cite[Theorem 1.1]{CET}}]
 There is a Lagrangian disk filling of $\leg_{+} $ if and only if  the contact $+1$-surgery on $\leg_{+} \subset \mathbb R^{3} \subset S^{3}$
produces a contact manifold that is  strongly symplectically fillable.
If $\leg_{+}$ has a decomposable Lagrangian filling, then the filling can be taken to be Stein. 
\end{theorem}

 In fact, \cite{CET} also shows that a  filling will be a Stein filling if and only if $\leg_{+}$ bounds a {\it regular} Lagrangian disk: 
a Lagrangian disk is regular if there is a Liouville vector field that is tangent to the disk.  Any decomposable Lagrangian filling is regular.

 We now see another strategy  to construct a non-decomposable Lagrangian filling.
 
 \begin{strat} Find a  Legendrian $\leg$ such that the
 $+1$-surgery on $\leg$ produces a contact manifold that is strongly symplectically fillable but does not admit a Stein filling.
 \end{strat}
 
 An issue with this approach is a lack of examples: there are very few manifolds which carry strongly fillable but not Stein fillable contact structures.  
 The main examples are  the $1/n$ surgeries on the positive and negative trefoils; see works by Ghiggini \cite{Ghi} and Tosun \cite{Tosun}.
  However it is not obvious whether any of these contact structures are a contact $+1$ surgery on a Legendrian knot in  $S^{3}$.

 \section{Conclusion}
 The desire to understand the flexibility and rigidity of Lagrangian submanifolds has led to a great deal of interesting research in symplectic topology.  
 Similarly, trying to understand constructions of and obstructions to Lagrangian cobordisms has led to many interesting results.  At this point, we 
 have few concrete answers to the Motivating Questions stated in our Introduction.  In particular, regarding Motivating Question~(\ref{ques:decomp}),
 there are presently many candidates 
for Legendrians $\leg_{\pm}$ that can be connected by a Lagrangian cobordism but not by a decomposable Lagrangian cobordism: by understanding
all the obstructions to Lagrangian cobordisms, one can come up with some good candidates. 
When trying and failing to construct a Lagrangian cobordism between a given pair, one may gain 
  intuition for additional obstructions to Lagrangian cobordisms that are waiting to be discovered.

\bibliographystyle{alpha}
\bibliography{Bibliography}

\newcommand{\etalchar}[1]{$^{#1}$}
\begin{thebibliography}{CDRGG20}

\bibitem[BLWar]{BLW}
John~A Baldwin, Tye Lidman, and C-M~Michael Wong.
\newblock Lagrangian cobordisms and {L}egendrian invariants in knot {F}loer
  homology.
\newblock {\em Michigan Mathematical Journal}, to appear.

\bibitem[BS18]{BSa}
John~A. Baldwin and Steven Sivek.
\newblock Invariants of {L}egendrian and transverse knots in monopole knot
  homology.
\newblock {\em J. Symplectic Geom.}, 16(4):959--1000, 2018.

\bibitem[BSar]{BS}
John~A Baldwin and Steven Sivek.
\newblock On the equivalence of contact invariants in sutured {F}loer homology
  theories.
\newblock {\em Geometry \& Topology}, to appear.

\bibitem[BST15]{bst:construct}
F.~Bourgeois, J.~Sabloff, and L.~Traynor.
\newblock Lagrangian cobordisms via generating families: {C}onstruction and
  geography.
\newblock {\em Algebr. Geom. Topol.}, 15(4):2439--2477, 2015.

\bibitem[BVVV13]{BVV}
John~A. Baldwin, David~Shea Vela-Vick, and Vera V\'{e}rtesi.
\newblock On the equivalence of {L}egendrian and transverse invariants in knot
  {F}loer homology.
\newblock {\em Geom. Topol.}, 17(2):925--974, 2013.

\bibitem[CDRGG]{CDRGG-concord}
Baptiste Chantraine, Georgios Dimitroglou~Rizell, Paolo Ghiggini, and Roman
  Golovko.
\newblock Floer homology and {L}agrangian concordance.
\newblock In {\em Proceedings of $21^{st}$ G\:okova Geometry-Topology
  Conference}, pages 76--113.

\bibitem[CDRGG20]{CDRGG-cobord}
Baptiste Chantraine, Georgios Dimitroglou~Rizell, Paolo Ghiggini, and Roman
  Golovko.
\newblock Floer theory for {L}agrangian cobordisms.
\newblock {\em J. Differential Geom.}, 114(3):393--465, 2020.

\bibitem[CE12]{CE}
Kai Cielieback and Yakov Eliashberg.
\newblock {\em From {S}tein to {W}einstein and back: {S}ymplectic geometry of
  affine complex manifolds}, volume~59.
\newblock AMS Colloquium Publications, 2012.

\bibitem[CETar]{CET}
James Conway, John~B. Etnyre, and B\"ulent Tosun.
\newblock Symplectic fillings, contact surgeries, and {L}agrangian disks.
\newblock {\em International Mathematics Research Notices}, to appear.

\bibitem[Cha10]{Cha}
Baptiste Chantraine.
\newblock Lagrangian concordance of {L}egendrian knots.
\newblock {\em Algebr. Geom. Topol.}, 10(1):63--85, 2010.

\bibitem[Cha12]{C:slice}
Baptiste Chantraine.
\newblock Some non-collarable slices of {L}agrangian surfaces.
\newblock {\em Bull. Lond. Math. Soc.}, 44(5):981--987, 2012.

\bibitem[Cha15a]{chantraine:disconnected-ends}
B.~Chantraine.
\newblock A note on exact {L}agrangian cobordisms with disconnected
  {L}egendrian ends.
\newblock {\em Proc. Amer. Math. Soc.}, 143(3):1325--1331, 2015.

\bibitem[Cha15b]{Cha2}
Baptiste Chantraine.
\newblock Lagrangian concordance is not a symmetric relation.
\newblock {\em Quantum Topol.}, 6(3):451--474, 2015.

\bibitem[Che02]{Che}
Yu.~V. Chekanov.
\newblock Invariants of {L}egendrian knots.
\newblock In {\em Proceedings of the {I}nternational {C}ongress of
  {M}athematicians, {V}ol. {II} ({B}eijing, 2002)}, pages 385--394. Higher Ed.
  Press, Beijing, 2002.

\bibitem[CN13]{ChNg}
Wutichai Chongchitmate and Lenhard Ng.
\newblock An atlas of {L}egendrian knots.
\newblock {\em Exp. Math.}, 22(1):26--37, 2013.

\bibitem[CNS16]{CNS}
Christopher Cornwell, Lenhard Ng, and Steven Sivek.
\newblock Obstructions to {L}agrangian concordance.
\newblock {\em Algebr. Geom. Topol.}, 16(2):797--824, 2016.

\bibitem[CSLL{\etalchar{+}}20]{WIG}
Orsola Capovilla-Searle, No\'emie Legout, Ma{\"y}lis Limouzineau, Emmy Murphy,
  Yu~Pan, and Lisa Traynor.
\newblock Obstructions to reversing {L}agrangian surgery in {L}agrangian
  fillings.
\newblock {\em In preparation}, 2020.

\bibitem[DR15]{DR1}
Georgios Dimitroglou~Rizell.
\newblock Exact {L}agrangian caps and non-uniruled {L}agrangian submanifolds.
\newblock {\em Ark. Mat.}, 53(1):37--64, 2015.

\bibitem[DR16]{DR}
G.~Dimitroglou~Rizell.
\newblock Lifting pseudo-holomorphic polygons to the symplectisation of
  {$P\times\mathbb{R}$} and applications.
\newblock {\em Quantum Topol.}, 7(1):29--105, 2016.

\bibitem[EES05]{EESu}
Tobias Ekholm, John Etnyre, and Michael Sullivan.
\newblock Non-isotopic {L}egendrian submanifolds in {$\Bbb R^{2n+1}$}.
\newblock {\em J. Differential Geom.}, 71(1):85--128, 2005.

\bibitem[EES09]{EES}
T.~Ekholm, J.~B. Etnyre, and J.~M. Sabloff.
\newblock A duality exact sequence for {L}egendrian contact homology.
\newblock {\em Duke Math. J.}, 150(1):1--75, 2009.

\bibitem[EG98]{EliGro}
Yasha Eliashberg and Misha Gromov.
\newblock Lagrangian intersection theory: finite-dimensional approach.
\newblock In {\em Geometry of differential equations}, volume 186 of {\em Amer.
  Math. Soc. Transl. Ser. 2}, pages 27--118. Amer. Math. Soc., Providence, RI,
  1998.

\bibitem[EGH00]{EGH}
Y.~Eliashberg, A.~Givental, and H.~Hofer.
\newblock Introduction to symplectic field theory.
\newblock Number Special Volume, Part II, pages 560--673. 2000.
\newblock GAFA 2000 (Tel Aviv, 1999).

\bibitem[EHK16]{EkHoKa}
Tobias Ekholm, Ko~Honda, and Tam\'{a}s K\'{a}lm\'{a}n.
\newblock Legendrian knots and exact {L}agrangian cobordisms.
\newblock {\em J. Eur. Math. Soc. (JEMS)}, 18(11):2627--2689, 2016.

\bibitem[Ekh12]{Ekholm}
T.~Ekholm.
\newblock Rational {SFT}, linearized {L}egendrian contact homology, and
  {L}agrangian {F}loer cohomology.
\newblock In {\em Perspectives in analysis, geometry, and topology}, volume 296
  of {\em Progr. Math.}, pages 109--145. Birkh\"auser/Springer, New York, 2012.

\bibitem[Eli98]{Eli}
Yakov Eliashberg.
\newblock Invariants in contact topology.
\newblock In {\em Proceedings of the {I}nternational {C}ongress of
  {M}athematicians, {V}ol. {II} ({B}erlin, 1998)}, number Extra Vol. II, pages
  327--338, 1998.

\bibitem[EP96]{EP}
Yakov Eliashberg and Leonid Polterovich.
\newblock Local {L}agrangian 2-knots are trivial.
\newblock {\em Annals of Math}, 144:1:61--76, 1996.

\bibitem[Etn05]{Etnyre}
John~B. Etnyre.
\newblock Legendrian and transversal knots.
\newblock In {\em Handbook of knot theory}, pages 105--185. Elsevier B. V.,
  Amsterdam, 2005.

\bibitem[FI04]{FI}
Dmitry Fuchs and Tigran Ishkhanov.
\newblock Invariants of {L}egendrian knots and decompositions of front
  diagrams.
\newblock {\em Mosc. Math. J.}, 4(3):707--717, 783, 2004.

\bibitem[FPar]{FP}
Stefan Friedl and Mark Powell.
\newblock Homotopy ribbon concordance and {A}lexander polynomials.
\newblock {\em Archiv der Mathematik}, to appear.

\bibitem[FR11]{f-r}
D.~Fuchs and D.~Rutherford.
\newblock Generating families and {L}egendrian contact homology in the standard
  contact space.
\newblock {\em J. Topol.}, 4(1):190--226, 2011.

\bibitem[Fuc03]{Fuchs}
Dmitry Fuchs.
\newblock Chekanov-{E}liashberg invariant of {L}egendrian knots: existence of
  augmentations.
\newblock {\em J. Geom. Phys.}, 47(1):43--65, 2003.

\bibitem[Ghi05]{Ghi}
Paolo Ghiggini.
\newblock Strongly fillable contact 3-manifolds without {S}tein fillings.
\newblock {\em Geom. Topol.}, 9:1677--1687, 2005.

\bibitem[Gil84]{Gil}
Patrick~M. Gilmer.
\newblock Ribbon concordance and a partial order on {S}-equivalence classes.
\newblock {\em Topology and its Applications}, 18(2):313 -- 324, 1984.

\bibitem[GJ19]{GJ}
Marco Golla and Andr\'{a}s Juh\'{a}sz.
\newblock Functoriality of the {EH} class and the {LOSS} invariant under
  {L}agrangian concordances.
\newblock {\em Algebr. Geom. Topol.}, 19(7):3683--3699, 2019.

\bibitem[GL20]{GL}
Onkar~Singh Gujral and Adam~Simon Levine.
\newblock Khovanov homology and cobordisms between split links.
\newblock {\em arXiv preprint arXiv:2009.03406}, 2020.

\bibitem[Gor81]{Gor}
C.~McA. Gordon.
\newblock Ribbon concordance of knots in the {$3$}-sphere.
\newblock {\em Math. Ann.}, 257(2):157--170, 1981.

\bibitem[Gro85]{Gro}
M.~Gromov.
\newblock Pseudoholomorphic curves in symplectic manifolds.
\newblock {\em Invent. Math.}, 82:307--347, 1985.

\bibitem[GSY20]{GSY}
Roberta Guadagni, Joshua~M. Sabloff, and Matthew Yacavone.
\newblock Legendrian satellites and decomposable cobordisms.
\newblock {\em In preparation}, 2020.

\bibitem[JT06]{lisa-jill}
J.~Jordan and L.~Traynor.
\newblock Generating family invariants for {L}egendrian links of unknots.
\newblock {\em Algebr. Geom. Topol.}, 6:895--933 (electronic), 2006.

\bibitem[Lin16]{Lin}
Francesco Lin.
\newblock Exact {L}agrangian caps of {L}egendrian knots.
\newblock {\em J. Symplectic Geom.}, 14(1):269--295, 2016.

\bibitem[LOSS09]{LOSS}
Paolo Lisca, Peter Ozsv\'{a}th, Andr\'{a}s~I. Stipsicz, and Zolt\'{a}n
  Szab\'{o}.
\newblock Heegaard {F}loer invariants of {L}egendrian knots in contact
  three-manifolds.
\newblock {\em J. Eur. Math. Soc. (JEMS)}, 11(6):1307--1363, 2009.

\bibitem[LZ19]{LZ}
Adam~Simon Levine and Ian Zemke.
\newblock Khovanov homology and ribbon concordances.
\newblock {\em Bull. Lond. Math. Soc.}, 51(6):1099--1103, 2019.

\bibitem[MOS09]{MOS}
Ciprian Manolescu, Peter Ozsv\'{a}th, and Sucharit Sarkar.
\newblock A combinatorial description of knot {F}loer homology.
\newblock {\em Ann. of Math. (2)}, 169(2):633--660, 2009.

\bibitem[MZer]{MZ}
Maggie Miller and Ian Zemke.
\newblock Knot {F}loer homology and strongly homotopy-ribbon concordances.
\newblock {\em Mathematical Research Letters}, To apper.

\bibitem[Ng01]{Nthesis}
Lenhard~Lee Ng.
\newblock {\em Invariants of {L}egendrian links}.
\newblock ProQuest LLC, Ann Arbor, MI, 2001.
\newblock Thesis (Ph.D.)--Massachusetts Institute of Technology.

\bibitem[NOT08]{NOT}
Lenhard Ng, Peter Ozsv\'{a}th, and Dylan Thurston.
\newblock Transverse knots distinguished by knot {F}loer homology.
\newblock {\em J. Symplectic Geom.}, 6(4):461--490, 2008.

\bibitem[NR13]{NR}
Lenhard Ng and Daniel Rutherford.
\newblock Satellites of {L}egendrian knots and representations of the
  {C}hekanov-{E}liashberg algebra.
\newblock {\em Algebr. Geom. Topol.}, 13(5):3047--3097, 2013.

\bibitem[NS06]{NS}
Lenhard~L. Ng and Joshua~M. Sabloff.
\newblock The correspondence between augmentations and rulings for {L}egendrian
  knots.
\newblock {\em Pacific J. Math.}, 224(1):141--150, 2006.

\bibitem[NT04]{NgTra}
L.~Ng and L.~Traynor.
\newblock Legendrian solid-torus links.
\newblock {\em J. Symplectic Geom.}, 2(3):411--443, 2004.

\bibitem[Oan15]{Oan}
Alexandru Oancea.
\newblock {\it {F}rom {S}tein to {W}einstein and back. {S}ymplectic geometry of
  affine complex manifolds} [book review of {MR}3012475].
\newblock {\em Bull. Amer. Math. Soc. (N.S.)}, 52(3):521--530, 2015.

\bibitem[OST08]{OST}
Peter Ozsv\'{a}th, Zolt\'{a}n Szab\'{o}, and Dylan Thurston.
\newblock Legendrian knots, transverse knots and combinatorial {F}loer
  homology.
\newblock {\em Geom. Topol.}, 12(2):941--980, 2008.

\bibitem[Pan17]{Pan1}
Yu~Pan.
\newblock The augmentation category map induced by exact {L}agrangian
  cobordisms.
\newblock {\em Algebr. Geom. Topol.}, 17(3):1813--1870, 2017.

\bibitem[PC05]{CP}
P.~E. Pushkar and Yu.~V. Chekanov.
\newblock Combinatorics of fronts of {L}egendrian links, and {A}rnold's
  4-conjectures.
\newblock {\em Uspekhi Mat. Nauk}, 60(1(361)):99--154, 2005.

\bibitem[Rit09]{Ritter}
Alexander~F. Ritter.
\newblock Novikov symplectic homology and exact {L}agrangian embeddings.
\newblock {\em Geom. Topol.}, 13(2):943--978, 2009.

\bibitem[Rud97]{rudolph:slice-tb}
L.~Rudolph.
\newblock The slice genus and the {T}hurston-{B}ennequin invariant of a knot.
\newblock {\em Proc. Amer. Math. Soc.}, 125(10):3049--3050, 1997.

\bibitem[Sab05]{Sa}
Joshua~M. Sabloff.
\newblock Augmentations and rulings of {L}egendrian knots.
\newblock {\em Int. Math. Res. Not.}, (19):1157--1180, 2005.

\bibitem[Sau04]{Sau}
Denis Sauvaget.
\newblock Curiosit\'{e}s lagrangiennes en dimension 4.
\newblock {\em Ann. Inst. Fourier (Grenoble)}, 54(6):1997--2020 (2005), 2004.

\bibitem[ST13]{ST:obstruct}
Joshua~M. Sabloff and Lisa Traynor.
\newblock Obstructions to {L}agrangian cobordisms between {L}egendrians via
  generating families.
\newblock {\em Algebr. Geom. Topol.}, 13:2733 -- 2797, 2013.

\bibitem[Tos20]{Tosun}
B\"{u}lent Tosun.
\newblock Tight small {S}eifert fibered manifolds with {$e_0= -2$}.
\newblock {\em Algebr. Geom. Topol.}, 20(1):1--27, 2020.

\bibitem[Tra01]{lisa:links}
Lisa Traynor.
\newblock Generating function polynomials for {L}egendrian links.
\newblock {\em Geom. Topol.}, 5:719--760, 2001.

\bibitem[Zem19]{Zem}
Ian Zemke.
\newblock Knot {F}loer homology obstructs ribbon concordance.
\newblock {\em Ann. of Math. (2)}, 190(3):931--947, 2019.

\end{thebibliography}
\end{document}